\pdfoutput=1

\documentclass[oneside, reqno, 11pt, a4paper]{amsart}
\usepackage[australian]{babel}

\AtBeginDocument{
  \DeclareSymbolFont{AMSb}{U}{msb}{m}{n}
  \DeclareSymbolFontAlphabet{\mathbb}{AMSb}
}
\DeclareFontFamily{U}{mathx}{\hyphenchar\font45}
\DeclareFontShape{U}{mathx}{m}{n}{<-> mathx10}{}
\DeclareSymbolFont{mathx}{U}{mathx}{m}{n}
\DeclareMathAccent{\widebar}{0}{mathx}{"73}

\usepackage[dvipsnames]{xcolor}
\usepackage{graphicx}
\usepackage{caption}
\usepackage{subcaption}

\usepackage{pgf}
\usepgflibrary{fpu}

\graphicspath{{./}}

\usepackage[a4paper, pdftex, left=2cm, top=2cm, right=2cm, bottom=2cm]{geometry}
\usepackage{changepage}

\usepackage[foot]{amsaddr}
\numberwithin{equation}{section}

\usepackage{url}
\usepackage{xurl}          
\usepackage{hyperref}
\hypersetup{
    breaklinks=true,
    bookmarksopen=true,
    pdftitle={Rotating bases for finite element exterior calculus},
    pdfauthor={Santiago Badia, Jordi Manyer, Antoine Marteau},
    pdfsubject={},
    colorlinks=true,
    linkcolor=blue,
    citecolor=blue,
    filecolor=blue,
    urlcolor=blue
}

\newcommand{\tab}[1]{Tab.~\ref{#1}}
\newcommand{\sect}[1]{Sect.~\ref{#1}}
\newcommand{\app}[1]{Apdx.~\ref{#1}}

\usepackage{csquotes}
\usepackage[backend=biber, defernumbers=true, maxbibnames=5, style=numeric-comp, isbn=false, bibencoding=utf8, safeinputenc, url=false, doi=true, giveninits=true]{biblatex}
\usepackage{mathtools}
\usepackage{mathabx}
\usepackage{stmaryrd}
\usepackage{dsfont}

\usepackage{float}
\usepackage{framed}
\usepackage{verbatim}
\usepackage{fancyvrb}
\usepackage{booktabs}
\usepackage{longtable}
\usepackage{colortbl}
\usepackage{multirow}
\usepackage{array}
\usepackage{cancel}
\usepackage{accents}
\usepackage{mleftright}
\usepackage{thm-restate}
\usepackage{etoolbox}
\usepackage[normalem]{ulem}

\usepackage[inline]{enumitem}
\usepackage{siunitx}

\newtheorem{theorem}{Theorem}
\newtheorem{lemma}{Lemma}
\newtheorem{corollary}{Corollary}

\newtheorem{definition}{Definition}
\theoremstyle{remark}
\newtheorem{remark}{Remark}
\theoremstyle{plain}

\usepackage{listings}
\lstdefinelanguage{julia}{%
  morekeywords={function,end,if,else,elseif,for,in,while,return,copy,sort,
                minimum,minmax,invperm,true,false,struct,const,let,do},
  sensitive=true,
  morecomment=[l]\#,
  morestring=[b]",
}
\usepackage{amsmath, amsfonts, amssymb, amscd, bm, mathtools}
\mathtoolsset{showonlyrefs}

\newcommand{\identity}{\mathds{1}}
\newcommand{\R}{\mathbb{R}}
\newcommand{\N}{\mathbb{N}}

\newcommand{\iv}[1]{\llbracket #1 \rrbracket}

\DeclareMathOperator{\supp}{supp}
\DeclareMathOperator{\curl}{curl}
\DeclareMathOperator{\trs}{tr}
\DeclareMathOperator{\spn}{span}

\newcommand{\PrL}{\mathcal{P}_r\Lambda^1}
\newcommand{\PrmL}{\mathcal{P}_r^{-}\Lambda^1}

\newcommand{\dd}{\mathop{}\!\mathrm{d}}   
\newcommand{\BB}{b}                       
\newcommand{\ind}[1]{\identity_{#1}}      
\newcommand{\sort}{\uparrow}              

\newcommand{\Sset}{\mathcal{S}}           
\newcommand{\Bset}{\mathcal{B}}           

\newcommand{\sects}[2]{Sects.~\ref{#1} and~\ref{#2}}

\newif\iftrackchanges  \trackchangesfalse
\newif\ifshowcomments  \showcommentsfalse

\usepackage{acronym}
\acrodef{pde}[PDE]{partial differential equation}
\acrodef{fe}[FE]{finite element}
\acrodef{fem}[FEM]{finite element method}
\acrodefplural{fem}[FEMs]{finite element methods}
\acrodef{dof}[DOF]{degree of freedom}
\acrodefplural{dof}[DOFs]{degrees of freedom}
\acrodef{feec}[FEEC]{finite element exterior calculus}

\title[Rotating bases for FEEC]{
Rotating bases for finite element exterior calculus: closed-form change of basis under vertex permutations
}
\date{\today}

\address{$^\dagger$School of Mathematics\\Monash University\\Clayton\\Victoria 3800\\Australia}
\author[S. Badia]{Santiago Badia$^{\dagger}$}
\email{santiago.badia@monash.edu}
\author[J. Manyer]{Jordi Manyer$^{\dagger}$}
\email{jordi.manyer@monash.edu}
\author[A. Marteau]{Antoine Marteau$^{\dagger}$}
\email{antoine.marteau@monash.edu}

\begin{document}

\begin{abstract}
The geometrically decomposed bases of the polynomial differential form spaces $\PrL$ and $\PrmL$ on a simplex are products of a barycentric scalar polynomial and a directional or Whitney 1-form. These bases depend on an ordering of the simplex vertices. On unstructured meshes, neighbouring cells need not agree on this ordering, making it non-trivial to achieve conformity in finite element software. Existing strategies compute degree-of-freedom transformation matrices numerically, or impose a global vertex ordering through face-bubble spaces or by preprocessing the mesh. We consider the standard basis of the $\PrmL$ (trimmed) space and introduce a new basis for the $\PrL$ (full) space, for arbitrary polynomial order $r \geq 1$ and dimension $D \geq 2$, and prescribe their basis polynomials as shape functions. We show that the pullback of shape functions under the change of coordinates induced by an arbitrary vertex relabelling $\pi$ in the symmetric group $S_{D+1}$ admits a closed-form, combinatorial expression. For most basis functions this pullback is a single basis function of the relabelled ordering, up to sign. On an explicitly characterised ``filter-hit'' set, this pullback is a signed sum of at most $D$ (full space) or exactly two (trimmed space) relabelled basis functions. All coefficients are in $\{-1, +1\}$ for all $r$, $D$ and $\pi$. The inverse transformation is obtained by computing the formulas at $\pi^{-1}$, so no numerical inversion of a change-of-basis matrix is needed. Conforming assembly and evaluation on simplicial meshes with arbitrary vertex orderings reduce to index manipulation. Our results are compared with the study of relabelling-invariant bases of Berchenko-Kogan and Licht, and verified in an open-source Julia implementation in the Gridap.jl library.
\end{abstract}

\maketitle

\section{Introduction}\label{sec:intro}

Finite element exterior calculus (FEEC) \cite{ArnoldFalkWinther2006} unifies the
classical compatible finite element families, such as the Lagrange, N\'ed\'elec, Raviart--Thomas and
Brezzi--Douglas--Marini elements. On simplicial meshes, the latter are equivalent to instances
of the spaces $\mathcal{P}_r \Lambda^K$ and $\mathcal{P}_r^{-} \Lambda^K$ of polynomial differential
$K$-forms, for suitable form degree $K$ and polynomial degree $r$. These spaces admit the \emph{geometric decomposition} of Arnold, Falk and
Winther \cite{ArnoldFalkWinther2009} (see also \cite{Licht2022,BerchenkoKogan2025extension}),
with explicit face-owned bases built from scalar barycentric monomials and
directional or Whitney 1-forms. Such bases depend on an
ordering of the simplex vertices and, as Arnold, Falk and Winther observe, there
is no canonical way to choose one \cite{ArnoldFalkWinther2009}. On a mesh, each cell
reads its vertices in the local ordering supplied by the mesh data structure, and
neighbouring cells need not agree on the ordering of a shared face.
Conformity requires matching form traces across that face,
computed on each side from ordering-dependent basis functions. At lowest order,
the discrepancy reduces to edge signs \cite{RognesKirbyLogg2009,LohiKettunen2021}.
At higher order the face-owned basis functions of the two cells are related by
a non-trivial change of basis, and no closed form is
available for it in the literature.
Several strategies to achieve conformity, as well as theoretical study of existence of relabelling-invariant bases for the considered spaces, were proposed in the literature. We review them next.

\subsection{Numerically computed transformation matrices}\label{sec:related}
This approach is the one used in modern codes such as FEniCS/Basix. Rognes, Kirby and Logg \cite{RognesKirbyLogg2009}
handle H(div)/H(curl) assembly with a sorted-vertex numbering convention plus
per-facet sign flips, which is a complete solution at lowest order. Scroggs,
Dokken, Richardson and Wells \cite{ScroggsDokkenRichardsonWells2022} replace
conventions by per-face degree-of-freedom permutations and transformation
matrices, valid for arbitrary degree and cell polytope. Scroggs and Wells
\cite{ScroggsWells2026} compute the change-of-basis matrix $T$ of the entire shape-function basis, for an arbitrary Ciarlet element. For that, the authors decompose the vertex permutation $\pi$ of the cell into face-wise elementary reflections and rotations, in order to compute the contribution to $T$ of each of them independently and cheaply. The degree-of-freedom basis transforms by $T^{-\top}$. The same architecture appears in MFEM
with sign flips encoded as negative local indices
\cite{MFEM2021} and, for a larger class of mappings, in the FInAT/Firedrake
transformation framework \cite{Kirby2018,KirbyMitchell2019}. That this machinery
remains delicate in practice is documented by the recent resolution of the ``sign
conflict'' for $hp$-N\'ed\'elec elements with hanging nodes in deal.II
\cite{KinnewigWickBeuchler2025}. All of these compute the transformation \emph{numerically},
per family and degree, at setup time. Since the matrices are precomputed on
the reference element, their runtime overhead is small. However, this forces using $T$ in every consumer of reference element tabulations, and introduces floating-point errors during the numerical computation of $T$ in the implementation.
This work introduces bases with closed-form and cheap formulas for $T$ and $T^{-1}$, such that no numerical error is introduced, and constructing $T$ and $T^{-1}$ is not necessary to apply their action.

\subsection{Global ordering conventions and orientation-embedded bases}
An alternative is not using element-wise fixed reference bases. Ainsworth and Coyle \cite{AinsworthCoyle2003} and
Zaglmayr \cite{Zaglmayr2006} parametrise hierarchical shape functions by \emph{global} vertex
numbers, so that conformity holds by construction. Notably, this is the mechanism employed by NGSolve. Fuentes, Keith, Demkowicz and Nagaraj
\cite{FuentesKeithDemkowiczNagaraj2015} expand the idea with
\emph{orientation-embedded} shape functions for several sequences of elements. The functions owned by each mesh face (edge, facet, \dots) take the local orientation as a parameter,
realised as a permutation of the face's coordinates. Their bases admit closed-form expressions, so their work is the closest prior to ours, but require working with labelling-dependent bases that are unions of face-owned bases. Instead, this work provides bases for the complete cell spaces that transform under any vertex relabelling by closed-form laws.

\subsection{Mesh reorientation}
Agelek, Anderson, Bangerth and Barth
\cite{AgelekAndersonBangerthBarth2017} study the preprocessing alternative, namely to
\emph{orient} the mesh cells so that local orderings are globally compatible.
Unfortunately, this strategy cannot be applied in general. For 3D
hexahedral meshes compatible orderings need not exist \cite{AgelekAndersonBangerthBarth2017},
and non-orientable manifolds (e.g. a M\"obius band or a Klein bottle) admit no simplex mesh with globally consistent orientation.
This strategy is mostly relevant to orientable simplicial meshes, in which each cell's vertices are sorted by global index.
Even when a global orientation exists, reorientation is a global and thus expensive mesh
operation,  it is hard to parallelise since consistency must be negotiated across
processor boundaries, and it must be re-run after every mesh modification.
The present paper makes preprocessing unnecessary, since
any local ordering is accepted as-is and no global orientation is ever
required.

\subsection{Symmetry, invariance, and Whitney-form combinatorics}
The action of the
symmetric group $S_{D+1}$ on FEEC spaces by vertex relabelling has been studied in two recent works. Licht
\cite{Licht2023symmetry} studies when $\mathcal{P}_r \Lambda^k$ and $\mathcal{P}_r^{-} \Lambda^k$ admit bases that
are \emph{invariant} under vertex relabelling, meaning that the pullback by the change-of-coordinates induced by each permutation maps each basis function to
$\pm$ another. The author classifies candidate invariant bases through the monomial
representation theory of the symmetric group, and constructs such bases using complex coefficients for certain degrees in dimensions two and three.
Berchenko-Kogan \cite{BerchenkoKogan2024} proves which degrees admit invariant bases modulo sign flip in
these dimensions. On the triangle, $\mathcal{P}_r^{-} \Lambda^1$ admits
such a basis if and only if $r \notin 3\N_0 + 2$, and $\mathcal{P}_r \Lambda^1$ if
and only if $r \notin 3\N_0$.
Both works study an existence question on a single simplex, and neither provides the expansion of a basis function in the relabelled basis when invariance is not achievable.
Besides, when invariant bases do exist, they are not always geometrically decomposed bases, the ones conformity requires \cite[Corollary 4.10]{BerchenkoKogan2024} and most of the results are restricted to dimensions two and three.
Different authors  have used combinatorics to study these bases. Bossavit \cite{Bossavit2002,RapettiBossavit2009} worked on the same trimmed spaces spanning families as \cite{BerchenkoKogan2025extension} and this paper. Proposition 3.5 of \cite{RapettiBossavit2009} records the linear relations between the family forms, identities that we will use to derive change-of-basis formulas. Christiansen and Rapetti \cite{ChristiansenRapetti2016} state that these spaces have ``natural spanning families'' but no natural bases. Lohi and Kettunen \cite{LohiKettunen2021,Lohi2022}
treat the orientation of lowest-order Whitney forms combinatorially (they are
alternating in their vertex indices) and implement higher order via local
interpolation solves. None of these uses the combinatorial identities to describe
the action of $S_{D+1}$ on a basis.

\tab{tab:strategies} summarises the comparison between the different strategies found in the literature, as well as that proposed in this paper.

\begin{table}[htbp]
  \centering
  \small\setlength{\tabcolsep}{5pt}
  \begin{tabular}{@{}lccccc@{}}
    \toprule
    \textbf{Strategy} & \textbf{Closed} & \textbf{Fixed ref.} & \textbf{No pre-} & \textbf{Exact} & \textbf{Non-orientable}\\
                      & \textbf{form}   & \textbf{basis}      & \textbf{processing} & \textbf{arithmetic} & \textbf{meshes}\\
    \midrule
    DOF transformation matrices \cite{ScroggsDokkenRichardsonWells2022,ScroggsWells2026} & no & yes & yes & no & yes\\
    Global ordering conventions \cite{AinsworthCoyle2003,Zaglmayr2006,FuentesKeithDemkowiczNagaraj2015} & --- & no & yes & --- & yes\\
    Mesh reorientation \cite{AgelekAndersonBangerthBarth2017} & --- & yes & no & --- & no\\
    Rotating bases (this work) & yes & yes & yes & yes & yes\\
    \bottomrule
  \end{tabular}
  \caption{Strategies to achieve conformity. ``Closed form'' refers to
    an explicit formula for the inter-element transformation of a fixed
    reference basis. Entries ``---'' indicate that the strategy avoids the
    transformation rather than representing it.
    Mesh reorientation is the only strategy that requires a globally
    consistent orientation, which does not exist on non-orientable manifolds.
    The other strategies work from local vertex indices alone.}
  \label{tab:strategies}
\end{table}

\subsection{Contributions and outline}\label{sec:contributions}

To the best of our knowledge, a closed-form, uniform-in-$r$,
combinatorial transformation law for a geometrically decomposed basis of
$\PrL$ and $\PrmL$ under arbitrary vertex relabellings has not
appeared in the literature.
This paper provides such a law for every degree $r \geq 1$ and every dimension $D$, for the standard basis of $\PrmL$ of \cite{ArnoldFalkWinther2009} and for a new basis of $\PrL$. We call them \emph{rotating bases}. Under an arbitrary vertex relabelling $\pi \in S_{D+1}$, most basis functions pull
back to a \emph{single} basis function of the relabelled basis with coefficient $\pm 1$,
except on a small \emph{filter-hit} subset, where they pull back to a
signed sum of at most $D$ (full space) or exactly two (trimmed space) relabelled
basis functions, uniformly in $r$, $D$ and
$\pi$. For $\PrL$, the degree independence of the calculus comes from a modification of the
directional 1-forms $\psi$ of \cite{ArnoldFalkWinther2009}. We weight the differentials of barycentric coordinates by the \emph{support
indicator} of the Bernstein multi-exponent $\alpha$ rather than by its value. This results in an $r$-independent family of directional forms, that we denote by catalogue, generating the bases using linear combinations with scalar barycentric polynomials for every order. Our contributions are the following.
\begin{enumerate}
\item \emph{Closed-form transformation laws.} For the proposed bases of both $\PrL$ and $\PrmL$, we prove pullback
  theorems (Theorems~\ref{thm:U3} and~\ref{thm:T3}) giving the exact expansion of every basis
  function under every vertex relabelling, together with the characterisation of
  the filter-hit set, a resummation identity (full space) and a two-term
  exponent-shift identity (trimmed space) for the hit cases (Theorem~\ref{thm:U1}, Lemma~\ref{lem:psi-resummation}, Theorem~\ref{thm:T1} and Theorem~\ref{thm:T2}).
\item \emph{Closed-form inverse using combinatorics.} The change-of-basis matrix $T = T(\pi)$ is
  integer, has at most $D$ non-zero entries per row for $\PrL$ and at most
  two for $\PrmL$, and satisfies
  $T(\pi)^{-1} = T(\pi^{-1})$, such that no numerical inversion of $T$ is needed.
\item \emph{Conformity on arbitrary meshes.} Trace uniqueness across a shared
  mesh face is guaranteed by the reference-simplex relabelling calculus
  (Theorem~\ref{thm:conformity-reduction} and Corollary~\ref{cor:conformity}). \mbox{$H \Lambda^1$-conforming} assembly and evaluation on simplicial meshes with arbitrary local vertex orderings then need no numerically computed transformation matrix, and neither a global ordering convention nor mesh preprocessing.
\item \emph{Validated open-source implementation.} The bases are implemented in the Gridap finite element ecosystem
  \cite{BadiaVerdugo2020,VerdugoBadia2022}, with numerical validation of the change-of-basis formulas and a curl--curl problem solved on meshes with arbitrary local vertex orderings
  (\sects{sec:implementation}{sec:numexp}).
\end{enumerate}

The results are presented for form order $K = 1$. The extension to $K \geq 2$ is expected to follow
by wedging $K$ directional 1-forms, as in the construction of
\cite{ArnoldFalkWinther2009}, and is left to future work. Since Berchenko-Kogan proved \cite{BerchenkoKogan2024} that pullback-invariant bases fail to exist for infinitely many polynomial degrees, the multi-term filter-hit cases cannot be avoided in general. The contribution of this paper is to resolve those cases in closed form.

The paper is organised as follows. \sect{sec:prelim} introduces notations, the relabelling
calculus, and the reduction of mesh conformity to reference relabellings.
\sects{sec:full}{sec:trimmed} construct the rotating bases of $\PrL$ and $\PrmL$ and
prove the change-of-basis theorems. \sect{sec:dofs} details the choice of degree-of-freedom basis that fits the change formulas. \sect{sec:implementation} describes the
implementation, \sect{sec:numexp} the numerical verification, and
\sect{sec:conclusions} draws conclusions. \app{sec:bernstein} records the Bernstein-normalised form of the trimmed calculus.

\section{Preliminaries}\label{sec:prelim}

This section fixes notation and records the two elementary facts on which
everything else rests: how vertex relabellings act on barycentric monomials and barycentric
differentials (\sect{sec:prelim-simplex}), and how inter-element conformity on a simplicial
mesh can be achieved using a vertex-relabelling change-of-basis in the reference simplex.

\subsection{The simplex, barycentric coordinates, and vertex relabellings}\label{sec:prelim-simplex}

Let $\Delta = [v_0, \dots, v_D] \subset \R^D$ be a non-degenerate $D$-simplex whose vertices carry
the labels $0, \dots, D$, where $[\dots]$ denotes the convex hull of a set of points. The barycentric coordinates $\xi = (\xi_0, \dots, \xi_D)$ are the affine
functions determined by $\xi_i (v_j) = \delta_{ij}$. They satisfy $\xi_i \geq 0$ on $\Delta$ and
$\sum_{i=0}^D \xi_i = 1$. A labelling of the vertices, equivalently a choice of $\xi$, is called a \emph{local ordering}. We identify a face of $\Delta$ with its set of vertex labels
$f \subseteq \{0, \dots, D\}$, with $\dim f = |f| - 1$ and smallest
label $\min f$. For a single label $k$ we abbreviate $f \setminus k \doteq f \setminus \{k\}$, promoting a label to the singleton set whenever a set operation demands it. Throughout,
faces satisfy $\dim f \geq 1$. The label set carries its natural total order.\footnote{A
face has no ordering data of its own, but inherits one from the ambient
order. This inherited order is what makes the smallest label $\min f$
and the increasing listing of a face canonical,
involving no choice. Equalities between label sets, such as the
covering conditions of \sects{sec:full}{sec:trimmed}, are equalities of
sets.} When a geometric entity is required, we identify the
label set with the convex hull of its vertices. This standard abuse of notation is used throughout when the context leaves no ambiguity.

A multi-exponent is a vector $\alpha \in \N_0^{D+1}$ with modulus $|\alpha| = \sum_i \alpha_i$ and
support $\supp(\alpha) = \{i : \alpha_i > 0\}$. Its \emph{support indicator} is the vector
$\ind{\alpha} \in \{0,1\}^{D+1}$ with entries $(\ind{\alpha})_i = \iv{\alpha_i > 0}$, where $\iv{P}$
denotes the Iverson bracket, equal to $1$ if the predicate $P$ holds and to $0$
otherwise.
The associated barycentric monomial is
\begin{equation}
  \BB(\xi; \alpha) = \xi^\alpha = \xi_0^{\alpha_0} \cdots \xi_D^{\alpha_D}.
\end{equation}

\begin{remark}[normalisation]\label{rem:normalisation}
  The Bernstein polynomials often used in the literature instead of barycentric monomials include the multinomial factor
  $|\alpha|! / (\alpha_0 ! \cdots \alpha_D !)$, which is invariant under any permutation of the entries
  of $\alpha$. Statements whose formulas only permute the entries of $\alpha$, which
  covers the calculus related to $\PrL$, therefore hold for the Bernstein basis.
  However, the calculus related to $\PrmL$ transfers one unit of degree between entries
  of $\alpha$. The multinomial factor then fails to cancel, introducing ratios of multinomials in the change matrices instead of $\pm 1$. We work
  with bare monomials throughout. The Bernstein counterpart of the trimmed
  calculus is worked out in \app{sec:bernstein}.
\end{remark}

\textbf{Vertex relabellings.} Let $\pi \in S_{D+1}$ act on the labels $\{0, \dots, D\}$, where
$S_{D+1}$ denotes the symmetric group of that label set, that is, the group of
bijections of $\{0, \dots, D\}$ onto itself. We consider a second copy $\Delta_\lambda$ of the simplex, carrying its own barycentric coordinates $\lambda_0, \dots, \lambda_D$ with $\lambda_i (v_j) = \delta_{ij}$, and write $\Delta_\xi$ for the original. The relabelling acts through the affine map $A_\pi : \Delta_\xi \to \Delta_\lambda$ defined by $A_\pi (v_i) = v_{\pi(i)}$, which sends the vertex labelled $i$ to the vertex labelled $\pi(i)$. Its pullback relates the two coordinate systems by
\begin{equation}\label{eq:relabel}
  A_{\pi^{-1}}^* \xi_i = \lambda_{\pi(i)}, \quad \text{equivalently} \quad A_\pi^* \lambda_i = \xi_{\pi^{-1}(i)},
\end{equation}
since $(A_{\pi^{-1}}^* \xi_i)(v_j) = \xi_i (v_{\pi^{-1}(j)}) = \delta_{\pi(i) j}$. $\Delta$ is used when the labelling does not matter.
The effect of relabelling can be characterised for all indexed data. On a label set it acts elementwise,\footnote{The relabellings introduced below act bijectively on the label set
but not monotonically. They map faces to faces without preserving the inherited
orders. The basis polynomial relabelling will thus rely on combinatorial tests such as $\pi(\min f) \neq \min \pi(f)$.}
\begin{equation}
  \pi(f) = \{\pi(i) : i \in f\}.
\end{equation}
On a vector of $\R^{D+1}$ indexed by the labels, such as a multi-exponent, it
relocates entries,
\begin{equation}\label{eq:perm-action}
  \pi(\alpha)_i = \alpha_{\pi^{-1}(i)},
\end{equation}
so that $\pi(\alpha)_{\pi(i)} = \alpha_i$ and $\ind{\pi(\alpha)} = \pi(\ind{\alpha})$.

These maps satisfy $A_\pi A_\tau = A_{\pi \tau}$ and $A_\pi^{-1} = A_{\pi^{-1}}$. Differential forms are transported by the corresponding
\emph{pullbacks}. We write $A_{\pi^{-1}}^*$ for the pullback taking a form on $\Delta_\xi$ to a form on $\Delta_\lambda$. By \eqref{eq:relabel}, it is
computed as the substitution $\xi_i \mapsto \lambda_{\pi(i)}$, $\dd \xi_i \mapsto \dd \lambda_{\pi(i)}$, and the
reverse direction is $A_\pi^*$.
On barycentric monomials, the pullback is purely combinatorial.

\begin{lemma}[relabelling of barycentric monomials]\label{lem:bernstein}
  For every multi-exponent $\beta$,
  \begin{equation}
    A_{\pi^{-1}}^* (\BB(\xi; \beta)) = \BB(\lambda; \pi(\beta)), \qquad A_\pi^* (\BB(\lambda; \beta)) = \BB(\xi; \pi^{-1}(\beta)).
  \end{equation}
\end{lemma}
\begin{proof}
  By \eqref{eq:relabel}, $\xi^\beta = \prod_i \xi_i^{\beta_i} = \prod_i \lambda_{\pi(i)}^{\beta_i}
  = \prod_j \lambda_j^{\beta_{\pi^{-1}(j)}} = \lambda^{\pi(\beta)}$. The second identity is the first with
  $\pi \leftarrow \pi^{-1}$.
\end{proof}

\textbf{Ambient representation of 1-forms.} We express 1-forms on $\Delta$ in the barycentric
differentials $\dd \xi_0, \dots, \dd \xi_D$, which are subject to the single relation
\begin{equation}\label{eq:ambient-relation}
  \sum_{i=0}^D \dd \xi_i = 0
\end{equation}
obtained by differentiating $\sum_i \xi_i = 1$.

\subsection{Polynomial differential forms and geometric decomposition}\label{sec:prelim-feec}

Let $\mathcal{P}_r (\Delta)$ be the polynomials of degree at most $r$ on $\Delta$. Using
$\sum_i \xi_i = 1$ to homogenise, $\mathcal{P}_r (\Delta) = \spn\{\xi^\alpha : |\alpha| = r\}$.\footnote{Each term of the canonical basis of total degree $r' < r$ is multiplied by $(\sum_i \xi_i)^{r - r'} = 1$, which raises each summand of the expanded product to exact degree $r$ without changing the function.} We always
work with exponents of exact modulus. The full polynomial 1-form space is
\begin{equation}\label{eq:full-space}
  \PrL (\Delta) = \Bigl\{\sum_{i=0}^D p_i \, \dd \xi_i : p_i \in \mathcal{P}_r (\Delta)\Bigr\}
  = \spn\{\xi^\alpha \, \dd \xi_i : |\alpha| = r, \, 0 \leq i \leq D\},
\end{equation}
with
$\dim \PrL = D \binom{r + D}{D}$ \cite{ArnoldFalkWinther2006}.

The Whitney 1-forms are, for labels $i \neq j$,
\begin{equation}\label{eq:whitney}
  \varphi(\xi; i, j) = \xi_i \, \dd \xi_j - \xi_j \, \dd \xi_i,
\end{equation}
alternating in $(i,j)$. The trimmed space of degree $r \geq 1$ can be characterised as
the span of the following family of Whitney forms with homogeneous polynomial coefficients of degree $r - 1$
\cite{ArnoldFalkWinther2006,RapettiBossavit2009,ChristiansenRapetti2016}:
\begin{equation}\label{eq:trimmed-space}
  \PrmL (\Delta) = \spn\{\xi^\alpha \, \varphi(\xi; i, j) : |\alpha| = r - 1, \, i \neq j\},
  \quad \dim \PrmL = r \binom{r + D}{D - 1}.
\end{equation}
This family is the Whitney forms basis for $r = 1$, but is not linearly independent for $r \geq 2$ and $D \geq 2$.

Both spaces $\PrL$ and $\PrmL$ are independent of the ordering. The geometric decomposition of Arnold, Falk and Winther \cite{ArnoldFalkWinther2009} (see also \cite{Licht2022,BerchenkoKogan2025extension}) decomposes the global conforming counterpart of both spaces into \emph{bubble} sub-spaces owned by faces $f$ of the mesh. Consequently, defining bases on a single simplex $\Delta$ is enough to characterise the globally conforming bases. The authors define bases of all bubble spaces constructed as the product of a barycentric monomial times a directional or Whitney 1-form indexed by data supported on $f$. We follow the same approach in this paper.

\subsection{From mesh conformity to reference relabellings}\label{sec:prelim-conformity}

Each simplex $K$ of the mesh is
parametrised by an affine bijection
$\Phi_K : \Delta \to K$. It depends on an ordered list of the vertices of $K$, which is supplied by the mesh data structure and determines the local ordering, that is, the coordinate system used in $\Delta$. This ordering is assumed
arbitrary. 1-form-valued reference fields are mapped from $\Delta$ to $K$ by the pushforward by $\Phi_K$, or equivalently by the pullback along its inverse, that is $u|_K = (\Phi_K)_* \hat{u} = (\Phi_K^{-1})^* \hat{u}$. Conformity in $H \Lambda^1$
requires the trace of the forms on each inter-element face $F$, that is, their pullback by the inclusion $F \hookrightarrow K$, to be single-valued. The following result characterises the conformity of piecewise mapped fields in terms of the vertex-relabelling calculus applied in the reference simplex.

\begin{theorem}[reduction of conformity to reference relabellings]\label{thm:conformity-reduction}
  Let $K_+ = \Phi_\xi (\Delta)$ and $K_- = \Phi_\lambda (\Delta)$ share the face $F = \Phi_\xi (f_+) = \Phi_\lambda (f_-)$,
  where $f_\pm$ are faces of $\Delta$. Choose $\pi \in S_{D+1}$ such that
  \begin{equation}\label{eq:pi-compat}
    \Phi_\lambda (v_{\pi(i)}) = \Phi_\xi (v_i) \quad \text{for all } i \in f_+
  \end{equation}
  (such $\pi$ exist, and satisfy $\pi(f_+) = f_-$). Then, for $\hat{u}_+, \hat{u}_- \in \Lambda^1 (\Delta)$,
  the traces on $F$ of $u_+ = (\Phi_\xi^{-1})^* \hat{u}_+$ and
  $u_- = (\Phi_\lambda^{-1})^* \hat{u}_-$ coincide if and only if
  \begin{equation}
    \trs_{f_-} \hat{u}_- = \trs_{f_-} (A_{\pi^{-1}}^* \hat{u}_+).
  \end{equation}
\end{theorem}
\begin{proof}
  The pullback operation commutes with the restriction to subsimplices, so
  \[ \trs_F u_+ = (\Phi_\xi|_{f_+}^{-1})^* \trs_{f_+} \hat{u}_+, \qquad
     \trs_F u_- = (\Phi_\lambda|_{f_-}^{-1})^* \trs_{f_-} \hat{u}_-. \]
  Hence $\trs_F u_+ = \trs_F u_-$ iff
  $\trs_{f_-} \hat{u}_- = \theta^* \trs_{f_+} \hat{u}_+$ with
  $\theta = \Phi_\xi|_{f_+}^{-1} \circ \Phi_\lambda|_{f_-} : f_- \to f_+$. By \eqref{eq:pi-compat},
  $\Phi_\xi = \Phi_\lambda \circ A_\pi$ on $f_+$, so $\theta = A_\pi^{-1}|_{f_-}$, and
  $\theta^* \trs_{f_+} \hat{u}_+ = \trs_{f_-} (A_{\pi^{-1}}^* \hat{u}_+)$.
\end{proof}

\begin{remark}
  The permutation $\pi$ in \eqref{eq:pi-compat} is only determined on $f_+$ and is arbitrary
  elsewhere. This is not an issue since the basis functions owned by a face $f \subseteq f_+$
  will be indexed by data supported on $f$ (\sects{sec:full}{sec:trimmed}), so the relabelling actions on them depend on $\pi$ only through its restriction to $f_+$.
\end{remark}

\section{The rotating basis for \texorpdfstring{$\PrL$}{Pr Lambda1}}\label{sec:full}

This section constructs a geometrically decomposed basis of $\PrL$ from a spanning set of
products of a barycentric monomial and a directional 1-form, and proves that the basis
transforms under every vertex relabelling $\pi \in S_{D+1}$ by an explicit signed index
map with at most $D$ terms, given by analytical index manipulation.

\subsection{Directional 1-forms and spanning set}\label{sec:full-spanning}

\begin{definition}[directional 1-form]\label{def:psi}
  Let $f$ be a face of $\Delta$ with $\dim f \geq 1$, let $k \in f$, and let
  $s \in \{0, 1\}^{D+1}$ be an indicator vector with $\supp(s) \cup \{k\} = f$. The
  associated directional 1-form is
  \begin{equation}\label{eq:psi}
    \psi(\xi; f, k, s) = \dd \xi_k - \frac{s_k}{|s|} \sum_{j \in f} \dd \xi_j.
  \end{equation}
\end{definition}

Since $\supp(s) \cup \{k\} = f$,
two possible shapes occur:
\begin{equation}\label{eq:psi-shapes}
  \psi(\xi; f, k, s) =
  \begin{cases}
    \dd \xi_k - \dfrac{1}{|f|} \displaystyle\sum_{j \in f} \dd \xi_j & \text{if } \supp(s) = f,\\[2ex]
    \dd \xi_k & \text{if } \supp(s) = f \setminus k.
  \end{cases}
\end{equation}

These directional forms let us define the following geometrically decomposed
spanning set of $\PrL$ by instantiating $s$ as the support indicator
$\ind{\alpha}$.

\begin{definition}[spanning set]\label{def:spanning}
  $\Sset(\Delta_\xi, r)$ is the set of triplets $(f, k, \alpha)$ with $\dim f \geq 1$, $k \in f$,
  $|\alpha| = r$ and $\supp(\alpha) \cup \{k\} = f$. To each triplet is associated the 1-form
  \begin{equation}\label{eq:w}
    w(\xi; f, k, \alpha) = \BB(\xi; \alpha) \, \psi(\xi; f, k, \ind{\alpha}).
  \end{equation}
  The first shape of \eqref{eq:psi-shapes}
  occurs iff $\alpha_k > 0$ and the second iff $\alpha_k = 0$.
\end{definition}

We identify each triplet with its 1-form and each set of triplets with the corresponding family of forms, writing for instance $\spn \Sset(\Delta_\xi, r)$.

\begin{remark}[contrast with the directional forms of \cite{ArnoldFalkWinther2009}]
  The directional forms of Arnold, Falk and Winther are weighted by the exponent
  \emph{values}, namely $\psi_k^{\alpha, f} = \dd \xi_k - (\alpha_k / |\alpha|) \sum_{j \in f} \dd \xi_j$
  \cite[eq.~(8.2)]{ArnoldFalkWinther2009}. Replacing the support indicator
  $(\ind{\alpha})_k / |\ind{\alpha}|$ by the weight $\alpha_k / |\alpha|$ preserves the two cases of \eqref{eq:psi-shapes}, but makes
  it fully depend on $\alpha$, including the order $r = |\alpha|$. On the other hand, the directional
  forms \eqref{eq:psi} constitute an $r$-independent \emph{catalogue}
  (Definition~\ref{def:catalog}) that is sufficient to assemble the polynomial set and basis. The upcoming relabelling theorems still apply to the bases from \cite{ArnoldFalkWinther2009} for two main reasons, that
  the weight commutes with relabelling,
  $(\ind{\pi(\alpha)})_{\pi(k)} = (\ind{\alpha})_k$ and
  $(\pi(\alpha))_{\pi(k)} / |\pi(\alpha)| = \alpha_k / |\alpha|$, and that the weights over the face
  sum to one.
\end{remark}

\begin{theorem}[spanning]\label{thm:spanning}
  $\spn \Sset(\Delta_\xi, r) = \PrL (\Delta)$ for every $r \geq 1$.
\end{theorem}
\begin{proof}
  The inclusion $\spn \Sset(\Delta_\xi, r) \subseteq \PrL (\Delta)$ is straightforward,
  since every $w(\xi; f, k, \alpha)$ is a combination of the generators
  $\xi^\alpha \, \dd \xi_j$ of \eqref{eq:full-space}. For the converse inclusion, by
  \eqref{eq:full-space} it suffices to express each form $\xi^\alpha \, \dd \xi_i$
  with $|\alpha| = r$ in $\spn \Sset$. Given such a form, we set
  $f = \supp(\alpha) \cup \{i\}$ and distinguish three cases.

  Assume first that $\alpha_i = 0$. Then $\supp(\alpha) = f \setminus i$, and $|f| \geq 2$ because
  $r \geq 1$ forces $\supp(\alpha) \neq \emptyset$. The second shape of \eqref{eq:psi-shapes} gives
  $\psi(\xi; f, i, \ind{\alpha}) = \dd \xi_i$, and therefore
  $\xi^\alpha \, \dd \xi_i = w(\xi; f, i, \alpha)$ belongs to the spanning set.

  Assume next that $\alpha_i > 0$ and $|f| = 1$. Then $\alpha = r \, 1_i$, and the
  ambient relation \eqref{eq:ambient-relation} yields
  $\xi^\alpha \, \dd \xi_i = - \sum_{j \neq i} \xi^\alpha \, \dd \xi_j$. Every summand on the
  right has $\alpha_j = 0$, so it is covered by the first case.

  Assume finally that $\alpha_i > 0$ and $|f| \geq 2$. Then $\supp(\alpha) = f$, and the
  first shape of \eqref{eq:psi-shapes} yields, after multiplication by $\xi^\alpha$ and
  rearrangement,
  \begin{equation}
    \xi^\alpha \, \dd \xi_i = w(\xi; f, i, \alpha) + \frac{1}{|f|} \xi^\alpha \sum_{j \in f} \dd \xi_j
    = w(\xi; f, i, \alpha) - \frac{1}{|f|} \sum_{j \notin f} \xi^\alpha \, \dd \xi_j,
  \end{equation}
  where the last equality uses \eqref{eq:ambient-relation} in the form
  $\sum_{j \in f} \dd \xi_j = - \sum_{j \notin f} \dd \xi_j$. Every term of the last sum has
  $j \notin f = \supp(\alpha)$, so it is covered by the first case, and the proof is
  complete.
\end{proof}

\begin{remark}[$r = 0$]
  We require $r \geq 1$ throughout. $\mathcal{P}_0 \Lambda^1$ admits no face-owned
  geometric decomposition for $D \geq 2$, because the edge-owned functions alone would number $\binom{D+1}{2} > D = \dim \mathcal{P}_0 \Lambda^1$. Note that
  $\Sset(\Delta_\xi, 0) = \emptyset$.
\end{remark}

\subsection{Rotating basis}\label{sec:full-basis}

\begin{definition}[basis filter]\label{def:filter}
  The directional form $\psi(\xi; f, k, s)$ is \emph{anchored} if
  \begin{equation}\label{eq:filter}
    s_i = 0 \quad \text{for all} \quad i < \min(f \setminus k).
  \end{equation}
  $\Bset(\Delta_\xi, r)$ is defined to be the set of triplets $(f, k, \alpha) \in \Sset(\Delta_\xi, r)$ whose
  directional factor $\psi(\xi; f, k, \ind{\alpha})$ is anchored.
\end{definition}

\begin{theorem}[filter characterisation]\label{thm:U1}
  The directional form $\psi(\xi; f, k, s)$ fails to be anchored if and only
  if $\supp(s) = f$ and $k = \min f$. In particular, a triplet
  $(f, k, \alpha) \in \Sset(\Delta_\xi, r)$ lies in $\Bset(\Delta_\xi, r)$ if and only if
  $\supp(\alpha) \neq f$ or $k \neq \min f$.
\end{theorem}
\begin{proof}
  If $\supp(s) \neq f$, then $\supp(s) = f \setminus k$, so
  $\min(f \setminus k) = \min \supp(s)$ and every $i \in \supp(s)$ satisfies
  $i \geq \min(f \setminus k)$, so that \eqref{eq:filter} holds. If $\supp(s) = f$ and
  $k > \min f$, then $\min(f \setminus k) = \min f = \min \supp(s)$ and the same argument
  applies. If $\supp(s) = f$ and $k = \min f$, then $s_{\min f} = 1$ with
  $\min f < \min(f \setminus k)$, violating \eqref{eq:filter}. The triplet statement follows
  with $s = \ind{\alpha}$, since $\supp(\ind{\alpha}) = \supp(\alpha)$.
\end{proof}

\begin{definition}[directional catalogue]\label{def:catalog}
  The \emph{directional catalogue} $\Psi(\Delta_\xi)$ is the family of anchored directional
  forms, indexed by their arguments $(f, k, s)$.
\end{definition}

\begin{remark}
  The catalogue is not a set since some of its elements may be
  identical. The second shape of \eqref{eq:psi-shapes} does not depend on $f$ or $s$,
  so any two faces containing $k$ contribute the same form $\dd \xi_k$. By Theorem~\ref{thm:U1}, for each face $f$ the catalogue contains the
  $2 \dim f + 1$ forms with $\supp(s) = f \setminus k$, $k \in f$, and with
  $\supp(s) = f$, $k \in f \setminus \min f$. The catalogue is finite and $r$-independent.
\end{remark}

\begin{lemma}[resummation of the directional forms]\label{lem:psi-resummation}
  Let $\supp(s) = f$ and $k \in f$. Then
  \begin{equation}\label{eq:psi-resummation}
    \psi(\xi; f, k, s) = - \sum_{v \in f \setminus k} \psi(\xi; f, v, s).
  \end{equation}
  Additionally, every summand in the right-hand side is anchored when $k = \min f$.
\end{lemma}
\begin{proof}
  Since $\supp(s) = f$, $(f, v, s)$ are valid parameters of $\psi$ for every $v \in f$. By
  \eqref{eq:psi-shapes}, $\psi(\xi; f, v, s) = \dd \xi_v - \frac{1}{|f|} \sum_{j \in f} \dd \xi_j$.
  Summing over $v \in f \setminus k$,
  \begin{equation}
    \sum_{v \in f \setminus k} \psi(\xi; f, v, s)
    = \sum_{v \in f \setminus k} \dd \xi_v - \frac{|f| - 1}{|f|} \sum_{j \in f} \dd \xi_j
    = - \dd \xi_k + \frac{1}{|f|} \sum_{j \in f} \dd \xi_j
    = - \psi(\xi; f, k, s).
  \end{equation}
  For $k = \min f$, each $(f, v, s)$ with $v \in f \setminus \min f$ is anchored by
  Theorem~\ref{thm:U1}.
\end{proof}

\begin{theorem}[resummation identity]\label{thm:U2}
  Let $(f, k, \alpha) \in \Sset(\Delta_\xi, r) \setminus \Bset(\Delta_\xi, r)$, i.e., $\supp(\alpha) = f$ and $k = \min f$.
  Then
  \begin{equation}\label{eq:resummation}
    w(\xi; f, k, \alpha) = - \sum_{v \in f \setminus k} w(\xi; f, v, \alpha),
  \end{equation}
  where every triplet $(f, v, \alpha)$ for $v \in f \setminus k$ belongs to
  $\Bset(\Delta_\xi, r)$.
\end{theorem}
\begin{proof}
  Multiply \eqref{eq:psi-resummation}, at $s = \ind{\alpha}$, by $\BB(\xi; \alpha)$. The summands belong to $\Bset$ since their directional forms are anchored as in Lemma~\ref{lem:psi-resummation}.
\end{proof}

\begin{lemma}[cardinality]\label{lem:count}
  For a face $f$ with $d = \dim f$, the triplets of $\Bset(\Delta_\xi, r)$ owned by $f$
  number $(r + 1) \binom{r - 1}{d - 1}$, and in total
  \begin{equation}\label{eq:count}
    \# \Bset(\Delta_\xi, r) = D \binom{r + D}{D} = \dim \PrL.
  \end{equation}
\end{lemma}
\begin{proof}
  Fix $f$ with $|f| = d + 1$. By Theorem~\ref{thm:U1} the elements of $\Bset$ owned by $f$
  are: (a) $\supp(\alpha) = f \setminus k$ with $k \in f$ arbitrary, giving $(d+1)$ choices of $k$
  times the number of multi-exponents $\alpha$ with $|\alpha| = r$ and $\supp(\alpha) = f \setminus k$, that is, of ways of writing $r$ as an ordered sum of $d$ positive integers, which is $\binom{r - 1}{d - 1}$; and (b) $\supp(\alpha) = f$ with $k \in f \setminus \min f$, giving
  $d$ choices of $k$ times the $\binom{r - 1}{d}$ multi-exponents with $|\alpha| = r$ and $\supp(\alpha) = f$.
  The absorption identity $d \binom{r - 1}{d} = (r - d) \binom{r - 1}{d - 1}$
  collapses this count to
  \[ (d + 1) \binom{r - 1}{d - 1} + d \binom{r - 1}{d} = (r + 1) \binom{r - 1}{d - 1}, \]
  which is the dimension of the space of 1-forms of degree $r$ on a $d$-simplex
  with vanishing trace on its boundary
  \cite[Corollary 5.2 and eq.~(3.7)]{ArnoldFalkWinther2009}.
  The geometric decomposition of \cite[Section 8]{ArnoldFalkWinther2009} expresses
  $\PrL$ as the direct sum of those spaces over the faces of $\Delta_\xi$, so summing
  over the $\binom{D + 1}{d + 1}$ faces of each dimension gives
  \[ \# \Bset(\Delta_\xi, r) = \sum_{d = 1}^{D} \binom{D + 1}{d + 1} \, (r + 1) \binom{r - 1}{d - 1} = \dim \PrL. \]
\end{proof}

\begin{corollary}[basis]\label{cor:basis}
  $\Bset(\Delta_\xi, r)$ is a basis of $\PrL (\Delta)$.
\end{corollary}
\begin{proof}
  By Theorem~\ref{thm:U2} every element of $\Sset \setminus \Bset$ lies in $\spn \Bset$, so
  $\spn \Bset = \spn \Sset = \PrL$ by Theorem~\ref{thm:spanning}. By Lemma~\ref{lem:count} the
  spanning family $\Bset$ has cardinality $\dim \PrL$, hence is a basis. In particular, distinct triplets of $\Bset(\Delta_\xi, r)$ give distinct forms.
\end{proof}

\begin{remark}
  The index set of $\Bset(\Delta_\xi, r)$ is that of the basis of $\PrL$ in \cite[Section 9]{ArnoldFalkWinther2009}, whose anchoring condition is \eqref{eq:filter}, and its cardinality is computed there as well; only the weight in the directional forms differs. Since the functions differ, spanning has to be established for the present basis, which is what Theorem~\ref{thm:U2} and Lemma~\ref{lem:count} do.
\end{remark}

\begin{lemma}[face ownership]\label{lem:trace}
  Let $(f, k, \alpha) \in \Sset(\Delta_\xi, r)$ and let $g$ be a face of $\Delta_\xi$ that does not
  contain $f$. Then $\trs_g w(\xi; f, k, \alpha) = 0$.
\end{lemma}
\begin{proof}
  Since $f = \supp(\alpha) \cup \{k\}$ is not contained in $g$, some label $i$ of $f$ lies outside $g$, and $\xi_i$ vanishes identically on $g$. If $i \in \supp(\alpha)$, the monomial $\BB(\xi; \alpha)$ vanishes on $g$ and so does the trace. Otherwise $i = k$ with $\alpha_k = 0$, the second shape of \eqref{eq:psi-shapes} gives $\psi(\xi; f, k, \ind{\alpha}) = \dd \xi_k$, and $\trs_g \dd \xi_k = 0$ because $\xi_k$ vanishes on $g$.
\end{proof}

\begin{remark}[extension operators]\label{rem:extension}
  For faces $f \subseteq g$, let $E_{f, g}$ be the barycentric extension of \cite[Section 2.3]{ArnoldFalkWinther2009}: an expression in the barycentric coordinates of $f$ is read as the same expression in those of $g$, so that $\trs_f E_{f, g} = \mathrm{id}$. Since \eqref{eq:w} involves only the labels of $f$, every basis function owned by $f$ is the barycentric extension of its trace on $f$. Those traces are the elements of $\Bset(f, r)$ whose face is $f$ itself. They are a subfamily of a basis of $\PrL(f)$, by Corollary~\ref{cor:basis} applied on the simplex $f$, hence linearly independent, and by Lemma~\ref{lem:trace} applied on $f$ they have zero trace on every proper face of $f$. They therefore lie in the bubble space of $f$, the subspace of $\PrL(f)$ with vanishing trace on $\partial f$, whose dimension $(r + 1) \binom{r - 1}{d - 1}$ equals their number by Lemma~\ref{lem:count}, so they form a basis of it. Together with Lemma~\ref{lem:trace},
  this makes $\Bset(\Delta_\xi, r)$ a geometrically decomposed basis in the sense of
  \cite[Section 4]{ArnoldFalkWinther2009}.
\end{remark}

\subsection{Pullback under vertex relabellings}\label{sec:full-pullback}

\begin{lemma}[rotation of $\psi$]\label{lem:psi-equivariance}
  For every admissible argument of Definition~\ref{def:psi} and every $\pi \in S_{D+1}$,
  \begin{equation}
    A_{\pi^{-1}}^* (\psi(\xi; f, k, s)) = \psi(\lambda; \pi(f), \pi(k), \pi(s)).
  \end{equation}
  In particular
  $A_{\pi^{-1}}^* (\psi(\xi; f, k, \ind{\alpha})) = \psi(\lambda; \pi(f), \pi(k), \ind{\pi(\alpha)})$,
  by the identity $\ind{\pi(\alpha)} = \pi(\ind{\alpha})$ of \sect{sec:prelim-simplex}.
\end{lemma}
\begin{proof}
  The pullback $A_{\pi^{-1}}^*$ maps $\dd \xi_k \mapsto \dd \lambda_{\pi(k)}$ and
  $\sum_{j \in f} \dd \xi_j \mapsto \sum_{j \in \pi(f)} \dd \lambda_j$. Moreover
  $\pi(s)_{\pi(k)} = s_k$ and $|\pi(s)| = |s|$, so the weight in \eqref{eq:psi} is
  unchanged.
\end{proof}

\begin{theorem}[rotation of the directional catalogue]\label{thm:psi-pullback}
  Let $\psi(\xi; f, k, s)$ be anchored and let $\pi \in S_{D+1}$. Then
  $(\pi(f), \pi(k), \pi(s))$ is an admissible argument of Definition~\ref{def:psi} and exactly one of the following
  holds.
  \begin{enumerate}
  \item \emph{Relabelled}. If $s_k = 0$ or $\pi(k) \neq \min \pi(f)$, then
    $\psi(\lambda; \pi(f), \pi(k), \pi(s))$ is anchored and
    \begin{equation}\label{eq:psi-pullback-single}
      A_{\pi^{-1}}^* (\psi(\xi; f, k, s)) = \psi(\lambda; \pi(f), \pi(k), \pi(s)).
    \end{equation}
  \item \emph{Filter hit.} If $s_k = 1$ and $\pi(k) = \min \pi(f)$, then
    \begin{equation}\label{eq:psi-pullback-hit}
      A_{\pi^{-1}}^* (\psi(\xi; f, k, s)) = - \sum_{v \in \pi(f) \setminus \pi(k)} \psi(\lambda; \pi(f), v, \pi(s)),
    \end{equation}
    and every summand on the right-hand side is anchored.
  \end{enumerate}
  In both cases the transformation stays inside the catalogue $\Psi(\Delta_\lambda)$, depends only on $(f, k, s)$ and $\pi$, and is independent of the degree $r$.
\end{theorem}
\begin{proof}
  Admissibility follows from $\pi(f) = \pi(\supp(s) \cup \{k\}) = \supp(\pi(s)) \cup \{\pi(k)\}$. By Lemma~\ref{lem:psi-equivariance}, the pullback is
  $\psi(\lambda; \pi(f), \pi(k), \pi(s))$. By Theorem~\ref{thm:U1} on $\Delta_\lambda$, this form fails to be anchored precisely when $\supp(\pi(s)) = \pi(f)$ and
  $\pi(k) = \min \pi(f)$, that is, when $s_k = 1$ and $\pi(k) = \min \pi(f)$. In that
  case Lemma~\ref{lem:psi-resummation} on $\Delta_\lambda$ applied to $\psi(\lambda; \pi(f), \pi(k), \pi(s))$
  gives \eqref{eq:psi-pullback-hit}. The summands are anchored by
  Theorem~\ref{thm:U1}.
\end{proof}

\begin{theorem}[pullback of the basis functions]\label{thm:U3}
  Let $(f, k, \alpha) \in \Bset(\Delta_\xi, r)$ and $\pi \in S_{D+1}$. Then
  \begin{equation}\label{eq:factorized}
    A_{\pi^{-1}}^* (w(\xi; f, k, \alpha))
    = \BB(\lambda; \pi(\alpha)) \, A_{\pi^{-1}}^* (\psi(\xi; f, k, \ind{\alpha})),
  \end{equation}
  and exactly one of the following holds.
  \begin{enumerate}
  \item \emph{Relabelled}. If $\alpha_k = 0$ or $\pi(k) \neq \min \pi(f)$, then
    \[ A_{\pi^{-1}}^* (w(\xi; f, k, \alpha)) = w(\lambda; \pi(f), \pi(k), \pi(\alpha)), \]
    with $(\pi(f), \pi(k), \pi(\alpha)) \in \Bset(\Delta_\lambda, r)$.
  \item \emph{Filter hit.} If $\alpha_k > 0$ and $\pi(k) = \min \pi(f)$, then
    \[ A_{\pi^{-1}}^* (w(\xi; f, k, \alpha)) = - \sum_{v \in \pi(f) \setminus \pi(k)} w(\lambda; \pi(f), v, \pi(\alpha)), \]
    with $(\pi(f), v, \pi(\alpha)) \in \Bset(\Delta_\lambda, r)$ for all $v$ of the sum.
  \end{enumerate}
\end{theorem}
\begin{proof}
  The pullback operation is multiplicative and relabels the barycentric scalar factor by
  Lemma~\ref{lem:bernstein}. The directional factor is transformed by Theorem~\ref{thm:psi-pullback}
  at $s = \ind{\alpha}$, whose two cases give the two cases of the statement,
  via $\ind{\pi(\alpha)} = \pi(\ind{\alpha})$. Using Theorem~\ref{thm:psi-pullback}, the directional forms associated to all right-hand-side triplets are anchored, and since $|\pi(\alpha)| = r$, they are in $\Bset(\Delta_\lambda, r)$.
\end{proof}

\begin{theorem}[same span and combinatorial inverse]\label{thm:U4}
  The families
  \[ \{A_{\pi^{-1}}^* w(\xi; f, k, \alpha) : (f, k, \alpha) \in \Bset(\Delta_\xi, r)\} \quad \text{and} \quad
     \{w(\lambda; f, k, \alpha) : (f, k, \alpha) \in \Bset(\Delta_\lambda, r)\} \]
  span the same space, namely $\PrL (\Delta)$. The change of basis in the reverse
  direction is obtained by applying Theorem~\ref{thm:U3} at $\pi^{-1}$.
\end{theorem}
\begin{proof}
  Write $W_\xi = \{w(\xi; \mu) : \mu \in \Bset(\Delta_\xi, r)\}$ and $W_\lambda = \{w(\lambda; \mu) : \mu \in \Bset(\Delta_\lambda, r)\}$. Theorem~\ref{thm:U3} gives $A_{\pi^{-1}}^* W_\xi \subseteq \spn W_\lambda$ and, with the roles of $\xi$ and $\lambda$ exchanged and $\pi$ replaced by $\pi^{-1}$, $A_\pi^* W_\lambda \subseteq \spn W_\xi$. Applying the linear map $A_{\pi^{-1}}^*$ to the second inclusion and using $A_{\pi^{-1}}^* A_\pi^* = (A_\pi A_{\pi^{-1}})^* = \mathrm{id}$,
  \[ W_\lambda = A_{\pi^{-1}}^* A_\pi^* W_\lambda \subseteq \spn A_{\pi^{-1}}^* W_\xi \subseteq \spn W_\lambda, \]
  so $\spn A_{\pi^{-1}}^* W_\xi = \spn W_\lambda = \PrL (\Delta)$, the last equality by Corollary~\ref{cor:basis} on $\Delta_\lambda$.
\end{proof}

\begin{corollary}[matrix form and sparsity]\label{cor:sparsity}
  Define $C_\psi (\pi)$ on the index set of the directional catalogue $\Psi(\Delta_\xi)$ entrywise. For $\mu = (f, k, s)$, let
\[
  C_\psi (\pi)_{\mu \nu} =
  \begin{cases}
    +1 & \text{if } \nu = (\pi(f), \pi(k), \pi(s)) \text{ and } \mu \text{ is relabelled in Theorem~\ref{thm:psi-pullback}},\\
    -1 & \text{if } \nu = (\pi(f), v, \pi(s)),\ v \in \pi(f) \setminus \pi(k), \text{ and } \mu \text{ is a filter hit},\\
    \phantom{+}0 & \text{otherwise},
  \end{cases}
\]
so that \eqref{eq:psi-pullback-single} and \eqref{eq:psi-pullback-hit} read
$A_{\pi^{-1}}^* \psi_\mu = \sum_{\nu} C_\psi (\pi)_{\mu \nu} \, \psi_\nu (\lambda)$. Then
  \begin{equation}
    C_\psi (\pi)_{\mu \nu} \in \{-1, 0, +1\},
  \end{equation}
  each row carrying a single entry $+1$ or at most $D$ entries $-1$. $C_\psi (\pi)$ is
  independent of the degree $r$.
\end{corollary}
\begin{proof}
  By Theorem~\ref{thm:psi-pullback}, the row of $(f, k, s)$ carries either a single entry
  \[ +1 \ \text{at } (\pi(f), \pi(k), \pi(s)), \quad \text{or } \dim f \leq D \text{ entries } -1 \ \text{at } (\pi(f), v, \pi(s)), \ v \in \pi(f) \setminus \pi(k). \]
  Nothing in the definition depends on $r$.
\end{proof}

\begin{corollary}[combinatorial inverse of the catalogue matrix]\label{cor:psi-inverse}
  For every $\pi \in S_{D+1}$, $C_\psi (\pi)^{-1} = C_\psi (\pi^{-1})$.
\end{corollary}
\begin{proof}
  Let $e_\nu$ denote the unit row vector at the index $\nu$, so that row $\nu$ of a matrix $C$ is $e_\nu C$. For a relabelled row, $e_{(f, k, s)} C_\psi (\pi) = e_{(\pi(f), \pi(k), \pi(s))}$, and that index is relabelled again under $\pi^{-1}$, because a hit there would require $\pi(s)_{\pi(k)} = s_k = 1$ and $\pi^{-1}(\pi(k)) = k = \min f$, which the anchoring of $(f, k, s)$ excludes by Theorem~\ref{thm:U1}, so $e_{(\pi(f), \pi(k), \pi(s))} C_\psi (\pi^{-1}) = e_{(f, k, s)}$. Let $\mu = (f, k, s)$ be a hit row, so that $s_k = 1$, $k \neq \min f$ and $\pi(k) = \min \pi(f)$. Then
  \[ e_\mu C_\psi (\pi) = - \sum_{v \in \pi(f) \setminus \pi(k)} e_{(\pi(f), v, \pi(s))}. \]
  Under $\pi^{-1}$, the index with $v = \pi(\min f)$ hits again, since $\pi^{-1}(v) = \min f$ and $\pi(s)_v = s_{\min f} = 1$, the latter because $s_k = 1$ means $k \in \supp(s)$, so the covering condition reads $\supp(s) \cup \{k\} = \supp(s) = f$ and thus $\min f \in \supp(s)$. Its row is therefore
  \[ e_{(\pi(f), \pi(\min f), \pi(s))} C_\psi (\pi^{-1}) = - \sum_{u \in f \setminus \min f} e_{(f, u, s)}, \]
  while the other indices are relabelled, $e_{(\pi(f), v, \pi(s))} C_\psi (\pi^{-1}) = e_{(f, \pi^{-1}(v), s)}$ for $v \in \pi(f) \setminus \{\pi(k), \pi(\min f)\}$, and these $\pi^{-1}(v)$ run over $f \setminus \{k, \min f\}$. Multiplying the first display by $C_\psi (\pi^{-1})$ on the right,
  \[ e_\mu C_\psi (\pi) \, C_\psi (\pi^{-1}) = - \sum_{u \in f \setminus \{k, \min f\}} e_{(f, u, s)} + \sum_{u \in f \setminus \min f} e_{(f, u, s)} = e_{(f, k, s)} = e_\mu. \]
  Hence $C_\psi (\pi) \, C_\psi (\pi^{-1}) = I$.
\end{proof}

\begin{corollary}[transformation matrix]\label{cor:cob}
  For $\pi \in S_{D+1}$, define $T(\pi)$ on the basis by
  \begin{equation}
    A_{\pi^{-1}}^* (w_\mu (\xi)) = \sum_{\nu \in \Bset(\Delta_\lambda, r)} T(\pi)_{\mu \nu} \, w_\nu (\lambda),
    \quad \mu \in \Bset(\Delta_\xi, r).
  \end{equation}
  Then $T(\pi)$ is well defined, its entries lie in $\{-1, 0, +1\}$, each row
  carries a single entry $+1$ or at most $D$ entries $-1$, and
  $T(\pi)^{-1} = T(\pi^{-1})$. Unlike the catalogue matrix $C_\psi (\pi)$, the matrix
  $T(\pi)$ depends on $r$, but only through its index set $\Bset(\Delta_\xi, r)$.
\end{corollary}
\begin{proof}
  $\Bset(\Delta_\lambda, r)$ is a basis by Corollary~\ref{cor:basis}, so the expansion exists and is
  unique, and $T(\pi)$ is well defined. Theorem~\ref{thm:U3} exhibits the expansion of every
  row, with a single entry $+1$ at $(\pi(f), \pi(k), \pi(\alpha))$ or $\dim f \leq D$ entries
  $-1$ at the triplets $(\pi(f), v, \pi(\alpha))$, $v \in \pi(f) \setminus \pi(k)$. Composing Theorem~\ref{thm:U3}
  at $\pi$ and at $\pi^{-1}$, as in Theorem~\ref{thm:U4}, gives $T(\pi) \, T(\pi^{-1}) = I$.
\end{proof}

\begin{remark}[practical computation]
  In practice the change of basis can be applied through the factorisation
  \eqref{eq:factorized}. The catalogue $\Psi(\Delta_\xi)$, the matrices $C_\psi (\pi)$, and the
  applicable case of Theorem~\ref{thm:psi-pullback} for each catalogue form can be precomputable once for a given $K$ and $D$. Per basis
  function, the case test reads the stored exponent $\alpha_k$ and compares
  $\pi(k)$ with $\min \pi(f)$, at small and constant cost. The scalar factor
  undergoes the index permutation $\alpha \mapsto \pi(\alpha)$. Only the directional factors
  require change-of-basis. Since the matrix $T(\pi)$ of
  Corollary~\ref{cor:cob} is known in closed form, it is not necessary to allocate it to apply its action.
\end{remark}

\begin{remark}[contrast with numerically computed transformations]
  The numerically computed transformation matrices mentioned in the introduction act on
  the assembled shape functions, treating the relocation of the
  multi-exponent and the pullback of the 1-form factors together. Their blocks thus
  grow with the degree, and the structure exhibited here is invisible to them.
  The factorisation \eqref{eq:factorized} separates the two. The family of barycentric monomials is mapped to itself by every relabelling, each monomial to a single monomial (Lemma~\ref{lem:bernstein}); the catalogue is not, and the whole change of basis is carried by the $r$-independent matrix $C_\psi (\pi)$ with entries in $\{-1, 0, +1\}$.
\end{remark}

\section{The rotating basis for \texorpdfstring{$\PrmL$}{Pr- Lambda1}}\label{sec:trimmed}

This section develops the counterpart of \sect{sec:full}. The structure is similar, since we use a spanning set of products of a barycentric monomial and a Whitney 1-form, a
combinatorial filter, and closed-form pullback formulas, with the following differences. Firstly, the Whitney form \eqref{eq:whitney} is
position-dependent and the scalar factor has degree $r - 1$. Secondly, we choose to index Whitney form by \emph{ordered pairs} of vertices. Since a relabelled Whitney form is another one modulo sign change, the catalogue is invariant modulo sign flip. Linear combinations arise only for
the basis functions, through shifts of the multi-exponents. Lastly, although this multi-exponent change is different for each basis function and often requires combining several pulled-back functions, it is nevertheless closed-form with exactly
two terms, uniformly in $r$, $D$ and $\pi$.
Three new notations are used throughout. We write $(\cdot)^{\sort}$ for the
\emph{sorting operator}, which lists a tuple or set of labels in increasing order.
For an ordered pair $e = (e_1, e_2)$
of distinct labels we write
\begin{equation}
  \varepsilon(e) \doteq 2 \iv{e_1 < e_2} - 1,
\end{equation}
so that $\varphi(\xi; e) = \varepsilon(e) \, \varphi(\xi; e^{\sort})$ by alternation. For labels $i \neq j$, define
the \emph{shift operator} $\rho$ that acts on multi-exponents by
\begin{equation}
  \rho(i, j) \circ \alpha \doteq \alpha - 1_i + 1_j.
\end{equation}

\subsection{Whitney forms and the spanning set}\label{sec:trimmed-spanning}

\begin{definition}[spanning set]\label{def:trimmed-spanning}
  $\Sset^-(\Delta_\xi, r)$ is the set of triplets $(f, e, \alpha)$ with $\dim f \geq 1$, an ordered
  pair $e = (e_1, e_2)$ of labels of $f$ with $e_1 \neq e_2$, $|\alpha| = r - 1$ and
  $\supp(\alpha) \cup \{e_1, e_2\} = f$. To each is associated the 1-form
  \begin{equation}\label{eq:trimmed-w}
    w(\xi; f, e, \alpha) = \BB(\xi; \alpha) \, \varphi(\xi; e_1, e_2),
  \end{equation}
  with $\varphi$ the Whitney 1-form \eqref{eq:whitney}, that is $\varphi(\xi; e_1, e_2) = \xi_{e_1} \dd \xi_{e_2} - \xi_{e_2} \dd \xi_{e_1}$.
\end{definition}

We identify triplets with 1-forms as in \sect{sec:full}.

\begin{remark}[ordered pairs]
  We deliberately do not impose $e_1 < e_2$ in the spanning set, in contrast to
  the convention of \cite{ArnoldFalkWinther2009}. Keeping the pair orders and the
  alternation $\varphi(\xi; e_2, e_1) = -\varphi(\xi; e_1, e_2)$ enables simplifying the upcoming relabelling theorem by using the $\varepsilon$ symbol to handle sign changes.
  Note also that $\dim f \geq 1$ is guaranteed by $\{e_1, e_2\} \subseteq f$. So as in the full
  space, vertices own no basis functions.
\end{remark}

\begin{theorem}[spanning]\label{thm:trimmed-spanning}
  $\spn \Sset^-(\Delta_\xi, r) = \PrmL (\Delta)$ for every $r \geq 1$.
\end{theorem}
\begin{proof}
  Every element of
  $\Sset^-(\Delta_\xi, r)$ is of the form $\xi^\alpha \varphi(\xi; i, j)$ with $|\alpha| = r - 1$, hence lies
  in $\PrmL$ by \eqref{eq:trimmed-space}. Conversely if $f, (i, j), \alpha$ is an admissible argument in Definition~\ref{def:trimmed-spanning}, then the generator $\xi^\alpha \varphi(\xi; i, j)$
  of \eqref{eq:trimmed-space} equals $w(\xi; f, (i, j), \alpha)$ since $f = \supp(\alpha) \cup \{i, j\}$.
\end{proof}

\subsection{Extraction of the basis}\label{sec:trimmed-basis}

\begin{definition}[basis filter]\label{def:trimmed-filter}
  A pair $(f, e)$ is \emph{anchored} if $e_1 < e_2$ and $e_1 = \min f$.
  $\Bset^-(\Delta_\xi, r) \subseteq \Sset^-(\Delta_\xi, r)$ consists of the triplets $(f, e, \alpha)$ whose
  pair is anchored.
\end{definition}

\begin{theorem}[filter characterisation]\label{thm:T1}
  A triplet $(f, e, \alpha) \in \Sset^-(\Delta_\xi, r)$ has anchored pair
  if and only if $e_1 < e_2$ and
  \begin{equation}\label{eq:trimmed-filter}
    (\ind{\alpha})_i = 0 \quad \text{for all} \quad i < e_1.
  \end{equation}
\end{theorem}
\begin{proof}
  If the pair is anchored, then $e_1 = \min f < e_2$, and every
  $i < e_1$ lies outside $f$, hence outside $\supp(\alpha) \subseteq f$, so both
  conditions hold. Conversely, assume $e_1 < e_2$ and \eqref{eq:trimmed-filter}.
  Every element of $\supp(\alpha)$ is then at least $e_1$, and $e_2 > e_1$, so
  $\min f = \min(\supp(\alpha) \cup \{e_1, e_2\}) = e_1$ and the pair is anchored.
\end{proof}

\begin{remark}[contrast with Theorem~\ref{thm:U1}]
  Unlike for the full space, the filter-hit condition of the trimmed space is independent of $\alpha$. This
  $\alpha$-independence propagates to the pullback (Theorem~\ref{thm:T3}), so that all basis
  functions sharing a pair $(f, e)$ transform alike.
\end{remark}

\begin{lemma}[triangle identity]\label{lem:pou}
  For any three distinct labels $a, b, c$,
  \begin{equation}\label{eq:pou}
    \xi_a \, \varphi(\xi; b, c) = \xi_b \, \varphi(\xi; a, c) - \xi_c \, \varphi(\xi; a, b).
  \end{equation}
\end{lemma}
\begin{proof}
  Expanding against \eqref{eq:whitney}, both sides equal
  $\xi_a \xi_b \dd \xi_c - \xi_a \xi_c \dd \xi_b$ after pairwise
  cancellation. The identity is equivalent to the cyclic relation
  $\xi_a \varphi(b, c) + \xi_b \varphi(c, a) + \xi_c \varphi(a, b) = 0$ from
  \cite[Prop.~3.5]{RapettiBossavit2009}.
\end{proof}

\begin{theorem}[two-term decomposition]\label{thm:T2}
  Let $(f, e, \alpha) \in \Sset^-(\Delta_\xi, r) \setminus \Bset^-(\Delta_\xi, r)$ with $e_1 < e_2$, and set
  $m \doteq \min f$. Then $m < e_1$, $\alpha_m > 0$ and
  \begin{equation}\label{eq:trimmed-decomp}
    w(\xi; f, (e_1, e_2), \alpha)
    = w(\xi; f, (m, e_2), \rho(m, e_1) \circ \alpha) - w(\xi; f, (m, e_1), \rho(m, e_2) \circ \alpha),
  \end{equation}
  and both triplets on the right-hand side belong to $\Bset^-(\Delta_\xi, r)$.
\end{theorem}
\begin{proof}
  Since the pair $(f, e)$ is not anchored, $m = \min f \notin \{e_1, e_2\}$, so the covering condition gives
  $\alpha_m > 0$. To obtain \eqref{eq:trimmed-decomp}, instantiate \eqref{eq:pou} at $(a, b, c) = (m, e_1, e_2)$, multiply
  by $\BB(\xi; \alpha) / \xi_m$, and use
  \begin{equation}
    \frac{\BB(\xi; \alpha)}{\xi_m} \cdot \xi_{e_1} = \BB(\xi; \rho(m, e_1) \circ \alpha), \quad
    \frac{\BB(\xi; \alpha)}{\xi_m} \cdot \xi_{e_2} = \BB(\xi; \rho(m, e_2) \circ \alpha),
  \end{equation}
  which are polynomial identities because $\alpha_m \geq 1$. The right-hand side triplets belong to $\Sset^-$ since $\rho$ preserves the multi-index moduli, and the index exchanges preserve the cover $f = \supp(\alpha) \cup e$. The new pairs are anchored since $m = \min f$.
\end{proof}

\begin{lemma}[cardinality]\label{lem:trimmed-count}
  A face $f$ with $d = \dim f$ owns $d \binom{r}{d}$ different triplets of $\Bset^-(\Delta_\xi, r)$, and the cardinality of the whole set is
  \begin{equation}\label{eq:trimmed-count}
    \# \Bset^-(\Delta_\xi, r) = r \binom{r + D}{D - 1} = \dim \PrmL.
  \end{equation}
\end{lemma}
\begin{proof}
  Fix $f$ with $|f| = d + 1$. By Definition~\ref{def:trimmed-filter} the triplets owned by $f$ have
  $e_1 = \min f$ and $e_2 \in f \setminus \min f$, giving $d$ choices of pair. For each, the
  multi-exponents completing an admissible triplet have $|\alpha| = r - 1$, must be positive on the $d - 1$ labels
  of $f \setminus \{e_1, e_2\}$, and may be positive on the two labels $e_1, e_2$. The map
  \[ \alpha \mapsto \beta \doteq \alpha - \sum_{i \in f \setminus \{e_1, e_2\}} 1_i \]
  is a bijection from these exponents onto the multi-exponents $\beta$ with
  $\supp(\beta) \subseteq f$ and $|\beta| = (r - 1) - (d - 1) = r - d$, and there are
  $\binom{(r - d) + d}{d} = \binom{r}{d}$ such $\beta$. This is the number of ways to
  distribute $r - d$ units over the $d + 1$ labels of $f$.
    The absorption identity $d \binom{r}{d} = r \binom{r - 1}{d - 1}$ rewrites this count as $r \binom{r - 1}{d - 1}$. Summing over the $\binom{D + 1}{d + 1}$ faces of dimension $d$ and using Vandermonde's identity $\sum_{d = 1}^{D} \binom{D + 1}{d + 1} \binom{r - 1}{d - 1} = \binom{r + D}{D - 1}$,
  \[ \# \Bset^-(\Delta_\xi, r) = r \sum_{d = 1}^{D} \binom{D + 1}{d + 1} \binom{r - 1}{d - 1} = r \binom{r + D}{D - 1}, \]
  which is $\dim \PrmL$ by \eqref{eq:trimmed-space}.
  
\end{proof}

The lemma is also a consequence of \cite[Theorem 6.1 and Corollary 5.2]{ArnoldFalkWinther2009}, where $\Bset^-(\Delta_\xi, r)$ appears as the basis of $\PrmL$ obtained from the geometric decomposition; the direct count above is shorter and yields the face-wise cardinalities explicitly.

\begin{corollary}[basis]\label{cor:trimmed-basis}
  $\Bset^-(\Delta_\xi, r)$ is a basis of $\PrmL (\Delta)$.
\end{corollary}
\begin{proof}
  Alternation maps the $e_1 > e_2$ spanning elements to their sorted negatives,
  and Theorem~\ref{thm:T2} expresses every sorted element of $\Sset^- \setminus \Bset^-$ in
  $\spn \Bset^-$. Hence $\spn \Bset^- = \spn \Sset^- = \PrmL$ by
  Theorem~\ref{thm:trimmed-spanning}, and Lemma~\ref{lem:trimmed-count} concludes by
  cardinality. In particular, distinct triplets of $\Bset^-(\Delta_\xi, r)$ give distinct forms.
\end{proof}

\begin{lemma}[face ownership]\label{lem:trimmed-trace}
  Let $(f, e, \alpha) \in \Sset^-(\Delta_\xi, r)$ and let $g$ be a face of $\Delta_\xi$ that does not
  contain $f$. Then $\trs_g w(\xi; f, e, \alpha) = 0$.
\end{lemma}
\begin{proof}
  As in Lemma~\ref{lem:trace}, if some $i \in \supp(\alpha)$ lies outside $g$ the Bernstein factor
  vanishes identically on $g$. Otherwise $\supp(\alpha) \subseteq g$, and since
  $f = \supp(\alpha) \cup \{e_1, e_2\}$ is not contained in $g$, at least one of the two
  labels $e_1, e_2$ lies outside $g$. For that label $\xi_{e_i}$ vanishes on $g$ and
  $\trs_g \dd \xi_{e_i} = 0$, so both terms of \eqref{eq:whitney} have vanishing
  trace.
\end{proof}

Restricted to $\Bset^-(\Delta_\xi, r)$, Lemma~\ref{lem:trimmed-trace} also follows from \cite[Section 9]{ArnoldFalkWinther2009}. The extension operators of Remark~\ref{rem:extension} are defined
verbatim here, with \eqref{eq:whitney} in place of \eqref{eq:psi} and Lemma~\ref{lem:trimmed-trace} in
place of Lemma~\ref{lem:trace}.

\subsection{Pullback under vertex relabellings}\label{sec:trimmed-pullback}

\begin{theorem}[rotation of the pair catalogue]\label{thm:phi-equivariance}
  For distinct labels $e_1 \neq e_2$ and $\pi \in S_{D+1}$, writing
  $\pi(e) = (\pi(e_1), \pi(e_2))$,
  \begin{equation}
    A_{\pi^{-1}}^* (\varphi(\xi; e_1, e_2)) = \varphi(\lambda; \pi(e_1), \pi(e_2))
    = \varepsilon(\pi(e)) \, \varphi(\lambda; \pi(e)^{\sort}).
  \end{equation}
\end{theorem}
\begin{proof}
  The pullback $A_{\pi^{-1}}^*$ maps $\xi_{e_i} \mapsto \lambda_{\pi(e_i)}$ and
  $\dd \xi_{e_i} \mapsto \dd \lambda_{\pi(e_i)}$, and the second equality is the alternation of
  $\varphi$.
\end{proof}

The Whitney forms thus form a finite catalogue, independent of the degree $r$,
on which every relabelling acts as a signed permutation. Unlike in
\sect{sec:full-pullback}, no combinations arise at the catalogue level.

\begin{theorem}[pullback of the basis functions]\label{thm:T3}
  Let $(f, e, \alpha) \in \Bset^-(\Delta_\xi, r)$, so that $e_1 = \min f < e_2$, and let
  $\pi \in S_{D+1}$. Write $\varepsilon = \varepsilon(\pi(e))$, $(e_1^\pi, e_2^\pi) = \pi(e)^{\sort}$ and
  $m^\pi = \min \pi(f)$. Then
  \begin{equation}\label{eq:trimmed-pullback-raw}
    A_{\pi^{-1}}^* (w(\xi; f, e, \alpha)) = \varepsilon \, \BB(\lambda; \pi(\alpha)) \, \varphi(\lambda; \pi(e)^{\sort}),
  \end{equation}
  and exactly one of the following holds.
  \begin{enumerate}
  \item \emph{Relabelled}. If $e_1^\pi = m^\pi$, equivalently if
    $\pi^{-1}(m^\pi) \in \{e_1, e_2\}$, the pair $(\pi(f), \pi(e)^{\sort})$ is anchored and
    \[ A_{\pi^{-1}}^* w(\xi; f, e, \alpha) = \varepsilon \, w(\lambda; \pi(f), \pi(e)^{\sort}, \pi(\alpha)), \]
    with $(\pi(f), \pi(e)^{\sort}, \pi(\alpha)) \in \Bset^-(\Delta_\lambda, r)$.
  \item \emph{Filter hit.} Otherwise, and independently of $\alpha$,
    \begin{equation}\label{eq:trimmed-pullback-hit}
      \begin{aligned}
        A_{\pi^{-1}}^* (w(\xi; f, e, \alpha)) = \varepsilon \bigl( &\, w(\lambda; \pi(f), (m^\pi, e_2^\pi), \rho(m^\pi, e_1^\pi) \circ \pi(\alpha))\\
        - &\, w(\lambda; \pi(f), (m^\pi, e_1^\pi), \rho(m^\pi, e_2^\pi) \circ \pi(\alpha)) \bigr),
      \end{aligned}
    \end{equation}
    with both triplets on the right in $\Bset^-(\Delta_\lambda, r)$.
  \end{enumerate}
\end{theorem}
\begin{proof}
  The pullback operation is multiplicative, relabels the scalar factors by
  Lemma~\ref{lem:bernstein}, and rotates the Whitney forms by Theorem~\ref{thm:phi-equivariance},
  which gives \eqref{eq:trimmed-pullback-raw}. $|\pi(\alpha)| = r - 1$ and
  $\pi(\supp(\alpha) \cup \{e_1, e_2\}) = \supp(\pi(\alpha)) \cup \{\pi(e_1), \pi(e_2)\}$ ensure that
  $(\pi(f), \pi(e)^{\sort}, \pi(\alpha))$ belongs to the spanning set. By Definition~\ref{def:trimmed-filter} on $\Delta_\lambda$, this
  triplet is in $\Bset^-(\Delta_\lambda, r)$ iff its pair is anchored, that is, iff
  $e_1^\pi = m^\pi$. Otherwise, the remaining statements result from applying Theorem~\ref{thm:T2} on $\Delta_\lambda$ to
  $(\pi(f), \pi(e)^{\sort}, \pi(\alpha))$, since the face minimum is $m^\pi$.
\end{proof}

\begin{remark}[$r = 1$, no hits]\label{rem:no-hits}
  For $r = 1$ the exponent vanishes, so $f = \{e_1, e_2\}$ and
  $\min \pi(f) = \min(\pi(e_1), \pi(e_2))$. No filter hit can occur, and Theorem~\ref{thm:T3}
  reduces to $A_{\pi^{-1}}^* \varphi(\xi; e) = \varepsilon(\pi(e)) \, \varphi(\lambda; \pi(e)^{\sort})$, the classical
  alternating sign rule for Whitney edge elements \cite{LohiKettunen2021}. In the
  language of \cite{BerchenkoKogan2024}, the lowest-order Whitney basis is invariant up
  to sign \cite[Lemma 6]{Licht2023symmetry}.
\end{remark}

\begin{theorem}[same span and combinatorial inverse]\label{thm:T4}
  The families
  \[ \{A_{\pi^{-1}}^* w(\xi; f, e, \alpha) : (f, e, \alpha) \in \Bset^-(\Delta_\xi, r)\} \quad \text{and} \quad
     \{w(\lambda; f, e, \alpha) : (f, e, \alpha) \in \Bset^-(\Delta_\lambda, r)\} \]
  both span $\PrmL (\Delta)$. The reverse change of basis is obtained
  by applying Theorem~\ref{thm:T3} at $\pi^{-1}$.
\end{theorem}
\begin{proof}
  As for Theorem~\ref{thm:U4}, with Theorem~\ref{thm:T3} and Corollary~\ref{cor:trimmed-basis} in place of Theorem~\ref{thm:U3} and Corollary~\ref{cor:basis}.
\end{proof}

\begin{corollary}[matrix form and sparsity]\label{cor:trimmed-sparsity}
  Define $C_\varphi (\pi)$ on the catalogue of Whitney forms with sorted pairs by
  $A_{\pi^{-1}}^* \varphi_\mu = \sum_{\nu} C_\varphi (\pi)_{\mu \nu} \, \varphi_\nu (\lambda)$. Then $C_\varphi (\pi)$ is a
  signed permutation matrix with $C_\varphi (\pi)^{-1} = C_\varphi (\pi^{-1})$, independent of
  the degree $r$.
\end{corollary}
\begin{proof}
  The statements restate Theorem~\ref{thm:phi-equivariance}, Theorem~\ref{thm:T3} and Theorem~\ref{thm:T4}.
\end{proof}

\begin{corollary}[transformation matrix]\label{cor:trimmed-cob}
  For $\pi \in S_{D+1}$, define $T(\pi)$ on the basis by
  \begin{equation}
    A_{\pi^{-1}}^* (w_\mu (\xi)) = \sum_{\nu \in \Bset^-(\Delta_\lambda, r)} T(\pi)_{\mu \nu} \, w_\nu (\lambda),
    \quad \mu \in \Bset^-(\Delta_\xi, r).
  \end{equation}
  Then $T(\pi)$ is well defined, each row
  contains a single non-zero entry $\varepsilon(\pi(e)) \in \{-1, +1\}$ or exactly two non-zero entries,
  $\varepsilon(\pi(e))$ and $-\varepsilon(\pi(e))$, and $T(\pi)^{-1} = T(\pi^{-1})$. As in \sect{sec:full}, the
  matrix depends on $r$ only through its index set $\Bset^-(\Delta_\xi, r)$.
\end{corollary}
\begin{proof}
  $\Bset^-(\Delta_\lambda, r)$ is a basis by Corollary~\ref{cor:trimmed-basis}, so the expansion is
  unique and $T(\pi)$ is well defined. Theorem~\ref{thm:T3} exhibits the expansion of every
  row, a single entry $\varepsilon(\pi(e))$ at $(\pi(f), \pi(e)^{\sort}, \pi(\alpha))$ or the two entries of
  \eqref{eq:trimmed-pullback-hit}. Composing Theorem~\ref{thm:T3} at $\pi$ and at $\pi^{-1}$ gives $T(\pi) \, T(\pi^{-1}) = I$.
\end{proof}

\begin{remark}[the standard basis]\label{rem:standard-basis}
  As noted in \sect{sec:trimmed-basis}, $\Bset^-(\Delta_\xi, r)$ is the geometrically decomposed basis of $\PrmL$ from \cite[Section 9]{ArnoldFalkWinther2009}. The signed two-term transformation law of Theorem~\ref{thm:T3} is therefore a property of
  the standard basis of the literature.
  To the best of our knowledge it has gone unremarked; \cite{ArnoldFalkWinther2009} constructs the bases and their decomposition but does not consider the pullback of basis functions under vertex relabellings.
\end{remark}

\begin{remark}[practical computation]
  In practice the change of basis can be applied through
  \eqref{eq:trimmed-pullback-raw} and \eqref{eq:trimmed-pullback-hit}. The sign $\varepsilon$, the sorted pair, and the anchoring
  test $\pi^{-1}(\min \pi(f)) \in \{e_1, e_2\}$ are pair-level data, independent of $r$
  and of $\alpha$. The scalar factor undergoes the index permutation
  $\alpha \mapsto \pi(\alpha)$ and, only in the hit case, the two shifts $\rho$.
\end{remark}

\section{Degrees of freedom}\label{sec:dofs}

In Ciarlet's definition, a finite element is a triple $(K, V, \Sigma)$, with $V$ a space of functions on the cell $K$ and $\Sigma$ a basis of its dual $V'$, the \emph{degree-of-freedom} (DOF) basis. The shape functions are the basis of $V$ dual to $\Sigma$; when they are not known in closed form, they are computed by a numerical change of basis from a prebasis of $V$. The construction of this paper is the dual one. We prescribe the pair $(K, \Bset)$, with $\Bset$ the rotating basis and $V = \spn \Bset$, and define the DOF basis $\Sigma = \{\sigma_p\}$ \emph{implicitly}, as the basis of $V'$ dual to $\Bset$, so that $\Bset$ is the shape-function basis by definition. When the functionals are needed, they are computed by a change of basis from a prebasis of $V'$, see \eqref{eq:dof-vandermonde}. Two properties make this the natural choice. The geometrically decomposed bases of \cite{ArnoldFalkWinther2009}, and the rotating bases of this paper, are directly usable as shape functions, and the closed form of the relabelling matrices of \sects{sec:full}{sec:trimmed} is specific to them.
Indexing either basis by its triplets, that is $p, q \in \Bset(\Delta_\xi, r)$
for $\PrL$ and $p, q \in \Bset^-(\Delta_\xi, r)$ for $\PrmL$, suppose there exists a family $\{\sigma_p\}$ in $(\PrL)'$ or $(\PrmL)'$, respectively, such that
\begin{equation}\label{eq:dof-duality}
  \sigma_p (w(\xi; q)) = \delta_{pq}, \quad \text{so that} \quad
  u = \sum_p \sigma_p (u) \, w(\xi; p),
\end{equation}
for all $u$ in the corresponding polynomial space.

\begin{lemma}[dual basis]
  The family $\{\sigma_p\}$ is a basis of the dual space $(\PrL)'$, respectively
  $(\PrmL)'$; equivalently, the degree-of-freedom set is unisolvent.
\end{lemma}
\begin{proof}
  If $\sum_p c_p \, \sigma_p = 0$, evaluation at $w(\xi; q)$ gives $c_q = 0$ for
  every triplet $q$ of the basis, so the family is linearly independent. Its
  cardinality is $\dim \PrL$ by \eqref{eq:count}, respectively $\dim \PrmL$ by
  \eqref{eq:trimmed-count}, so it is a basis of the dual space.
\end{proof}

Since the definition \eqref{eq:dof-duality} is implicit, the functionals are implemented using a numerical change of basis. For any family $\ell = (\ell_q)$ of $\dim \PrL$
(respectively $\dim \PrmL$) functionals unisolvent on the space, the matrix
$M = (\ell_p (w(\xi; q)))_{pq}$ is invertible and
\begin{equation}\label{eq:dof-vandermonde}
  \sigma_p = \sum_q (M^{-1})_{pq} \, \ell_q.
\end{equation}
$M$ is assembled and inverted once per reference simplex.
We take for $\ell$ the
canonical degrees of freedom of finite element exterior calculus, of the form $u \to \int_f (\trs_f u) \wedge q$, the
integral moments of form traces against polynomial spaces on the
faces \cite{ArnoldFalkWinther2006}. Applied to a field $u$ external to the space,
$\sigma_p (u) = \sum_q (M^{-1})_{pq} \, \ell_q (u)$ then matches every canonical
moment of $u$, so the induced interpolation operator is the canonical
interpolant of finite element exterior calculus, commuting diagram
included, expressed in the rotating basis. The basis, the matrices $T(\pi)$
and the conformity enforcement do not depend on the choice of $\ell$.

To note, the trace properties of the geometric decomposition force $M$ to
be block upper triangular, where the blocks correspond to the non-empty
face bubble spaces. However, the DOF bases proposed in \cite{ArnoldFalkWinther2009}
are not dual to the bases they introduce, so $M$ is not trivial. It needs to
be computed in practice, either to implement the DOFs as in
\eqref{eq:dof-vandermonde}, or the shape functions. Let us also remark that the
relabelling formulas introduced in this paper can also be used to improve the
approaches such as \cite{ScroggsWells2026} where DOF are prescribed. Indeed, using
the DOF from \cite{ArnoldFalkWinther2009} together with the rotating bases for the
moment's test polynomials $q$, it would be possible to analytically compute
the relabelling relations for the DOF functionals.

In practice, the degree-of-freedom functionals are mostly used for interpolation of user-provided fields as boundary or
initial data. Assembly, evaluation and error computation only require shape
functions. When shape functions are prescribed in closed form, neither $M$ nor any DOF ever affects a FEM
computation, in contrast with the usual implementations, based on Ciarlet's construction, described in the introduction, where the
transformation applied during assembly is itself computed from the
DOF functionals.

\section{Conforming assembly on a mesh}\label{sec:conformity}

By Theorem~\ref{thm:conformity-reduction}, the conformity of a piecewise polynomial space is a statement about relabellings of the reference simplex. This section makes it explicit for the rotating bases: each cell carries the basis written in the barycentric coordinates sorted by global vertex index, and the identification of basis functions across cells is defined by face and position alone.

The mesh supplies for each cell $K$ the local-to-global vertex index map $\operatorname{ltg}_K$, with $\operatorname{ltg}_K(i) = g_i$ the global vertex index of the vertex of $K$ corresponding to $v_i$. Let $\pi_K \in S_{D+1}$ be the permutation \emph{consistent with the global vertex index}, that is, such that
\[ \pi_K(\operatorname{ltg}_K) = (g_{\pi_K^{-1}(0)}, \dots, g_{\pi_K^{-1}(D)}) \]
is increasing, with the action \eqref{eq:perm-action} of permutations on vectors indexed by labels. Then $\lambda_{\pi_K(i)} = \xi_i$, so that $\lambda$ is obtained from $\xi$ by $\pi_K$ as in \sect{sec:prelim-simplex}, $A_{\pi_K} : \Delta_\xi \to \Delta_\lambda$. The physical map of $K$ is the affine map
\begin{equation}\label{eq:geometric-map}
  \Phi_K : \Delta_\lambda \to K
\end{equation}
sending the $j$-th vertex of $\Delta_\lambda$ to the vertex of $K$ with the $j$-th smallest global vertex index; in local labels, $\Phi_K \circ A_{\pi_K} : \Delta_\xi \to K$ sends $v_i$ to the vertex of global index $g_i$, which is the cell map supplied by the mesh, \sect{sec:prelim-simplex}. The \emph{rotated basis} of $K$ is the pushforward of $\{w(\lambda; \nu) : \nu \in \Bset(\Delta_\lambda, r)\}$ by $\Phi_K$. Since affine maps preserve barycentric coordinates, it is the geometrically decomposed basis written in the barycentric coordinates $\lambda_K = \lambda \circ \Phi_K^{-1}$ of $K$, labelled by increasing global vertex index. It is defined by the global vertex indices alone.

The rotated basis is computed from the reference element, built in the coordinates $\xi$ of $\Delta$, by pulling back through $\Phi_K \circ A_{\pi_K}$: the pushforward by $\Phi_K$ of $w(\lambda; \nu)$ is the pushforward by $\Phi_K \circ A_{\pi_K}$ of $A_{\pi_K}^* w(\lambda; \nu)$, and on the reference simplex
\begin{equation}\label{eq:rotated-in-reference}
  A_{\pi_K}^* \, w(\lambda; \nu) = \sum_{\mu} T(\pi_K^{-1})_{\nu \mu} \, w(\xi; \mu),
\end{equation}
by Theorems~\ref{thm:U3} and~\ref{thm:T3} applied at $\pi_K^{-1}$ (Corollaries~\ref{cor:cob} and~\ref{cor:trimmed-cob}). Every rotated basis function is a signed combination of at most $D$ (full space) or two (trimmed space) reference shape functions, with coefficients the row $\nu$ of $T(\pi_K^{-1})$. No floating-point computation is involved; the only remaining step is the physical map $\Phi_K$.

For a face $f$ of $\Delta$, let $\Bset_f$ be the set of pairs $\nu$, $\nu = (k, \alpha)$ in the full space and $\nu = (e, \alpha)$ in the trimmed space, such that $(f, \nu) \in \Bset(\Delta_\lambda, r)$. With $f$ fixed, $\nu$ labels the basis functions owned by $f$ uniquely, and we write $(f, \nu)$ for the function. For faces $f, f'$ of equal dimension let $\sigma_{f \to f'} : f \to f'$ be the unique increasing bijection between the two label sets and let $\iota_{f \to f'} : \Bset_f \to \Bset_{f'}$, $\nu \mapsto \sigma_{f \to f'}(\nu)$, relabel $\nu$ by $\sigma_{f \to f'}$ acting as in \eqref{eq:perm-action}; it is a bijection, since increasing maps preserve minima and hence anchoring, and $\iota_{f' \to f''} \circ \iota_{f \to f'} = \iota_{f \to f''}$. A \emph{position map} is a family of bijections $p_f : \Bset_f \to \{0, \dots, \# \Bset_f - 1\}$ such that
\begin{equation}\label{eq:position-compat}
  p_{f'} \circ \iota_{f \to f'} = p_f \qquad \text{for all faces } f, f' \text{ of equal dimension}.
\end{equation}
We choose $p_{f_d}$ arbitrarily for a reference face $f_d$ of each dimension $d$ and set $p_f \doteq p_{f_d} \circ \iota_{f \to f_d}$; then $p_{f'} \circ \iota_{f \to f'} = p_{f_d} \circ \iota_{f' \to f_d} \circ \iota_{f \to f'} = p_{f_d} \circ \iota_{f \to f_d} = p_f$.

For a face $\hat f$ of $\Delta$ in the local labels of $K$, $\operatorname{ltg}_K(\hat f)$ is the set of global vertex indices of the corresponding face of $K$, which identifies that mesh face; Gridap and many other finite element codes store it as the local-to-global map $\operatorname{ltg}^d_K$ of the $d$-dimensional faces, $\operatorname{ltg}^0_K = \operatorname{ltg}_K$. The rotated function $(f, \nu)$ of $K$, $f$ in sorted labels, is owned by the mesh face $\operatorname{ltg}_K(\pi_K^{-1}(f)) = \{(\pi_K(\operatorname{ltg}_K))_i : i \in f\}$.

\begin{definition}[global identification]\label{def:identification}
    Let $(f_+, \nu_+)$ and $(f_-, \nu_-)$ be rotated functions of the cells $K_+$ and $K_-$. They are \emph{identified}, $(f_+, \nu_+) \sim (f_-, \nu_-)$, if
  \begin{equation}\label{eq:identification}
    \operatorname{ltg}_{K_+}(\pi_+^{-1}(f_+)) = \operatorname{ltg}_{K_-}(\pi_-^{-1}(f_-)) \qquad \text{and} \qquad p_{f_+}(\nu_+) = p_{f_-}(\nu_-).
  \end{equation}
  
\end{definition}

The first condition states that both functions are owned by the same mesh face $F$, the second fixes the face-wise ordering. Equivalently, $(f, \nu)$ receives the global DOF index $(\operatorname{ltg}^d_K(\pi_K^{-1}(f)), p_f(\nu))$, which is the index of the global shape function by \eqref{eq:dof-duality}, and identified functions are those with the same global DOF index. The assembled space is spanned by the global basis functions obtained by gluing the rotated functions according to \eqref{eq:identification}.

\begin{corollary}[conformity of the rotated bases]\label{cor:conformity}\label{cor:assembly}
    Let $K_\pm$ share the mesh face $F$.
  \begin{enumerate}
  \item A rotated function of $K_+$ owned by a mesh face $G \subseteq F$ is identified with exactly one rotated function of $K_-$, and that function is owned by $G$.
  \item Identified rotated functions have the same trace on $F$.
  \item Rotated functions of $K_\pm$ owned by faces not contained in $F$ have zero trace on $F$.
  \end{enumerate}
  Consequently, the space assembled by \eqref{eq:identification} is $H \Lambda^1$-conforming.
  
\end{corollary}
\begin{proof}
    (1) Let $g_\pm$ be the sorted labels of $G$ in $K_\pm$, so that $\operatorname{ltg}_{K_\pm}(\pi_\pm^{-1}(g_\pm)) = G$. By \eqref{eq:identification}, $(g_+, \nu_+)$ is identified with the functions $(g_-, \nu_-)$ of $K_-$ such that $p_{g_-}(\nu_-) = p_{g_+}(\nu_+)$, and with no function owned by another face. Since $p_{g_-}$ is a bijection onto $\{0, \dots, \# \Bset_{g_-} - 1\}$ and $\# \Bset_{g_-} = \# \Bset_{g_+}$, there is exactly one such $\nu_-$.
  (2) Both $g_+$ and $g_-$ list the vertices of $G$ by increasing global vertex index, hence the increasing bijection $\sigma_{g_+ \to g_-}$ maps the label of each vertex of $G$ in $K_+$ to its label in $K_-$. The barycentric coordinate of a vertex of $G$ restricted to $F$ is intrinsic to $F$, so $\lambda^+_a = \lambda^-_{\sigma_{g_+ \to g_-}(a)}$ on $F$ for every $a \in g_+$. Since \eqref{eq:w} and \eqref{eq:trimmed-w} involve only the labels of the owning face, $(g_+, \nu_+)$ and $(g_-, \iota_{g_+ \to g_-}(\nu_+))$ have the same trace on $F$. This is the function of (1), owned by $G$ and $p_{g_-}(\iota_{g_+ \to g_-}(\nu_+)) = p_{g_+}(\nu_+)$ by \eqref{eq:position-compat}.
  (3) This is Lemmas~\ref{lem:trace} and~\ref{lem:trimmed-trace}.
  Consequently, the trace on $F$ of a glued basis function is either zero on both sides, by (3), or the common trace of two identified functions, by (1) and (2); it is single valued on every mesh face.
  
\end{proof}

\section{Implementation in Gridap.jl}\label{sec:implementation}

This section describes how the closed forms from Theorems~\ref{thm:U3} and~\ref{thm:T3} are used in the open-source Julia implementation built on the Gridap finite element ecosystem \cite{BadiaVerdugo2020,VerdugoBadia2022}.

\subsection{Rotating bases implementation}\label{sec:impl-bases}

Both rotating bases are implemented as subtypes of Gridap's
\lstinline|PolynomialBasis|, the abstraction Gridap uses for every reference shape-function
family.
A basis instance stores Gridap's scalar Bernstein basis together with per-basis-function index
tables. For the full space these are the triplet table $(f, k, \alpha)$ of
Definition~\ref{def:spanning}, the identifier of $\alpha$ in the scalar Bernstein basis, and the
precomputed directional forms
$\psi(\xi; f, k, \ind{\alpha})$.
For the trimmed
space, the table stores the pair $(e_1, e_2)$ and the two constant 1-forms
$\dd \xi_{e_1}$, $\dd \xi_{e_2}$. Evaluation computes
$b(\xi; \alpha)(\xi_{e_1} \dd \xi_{e_2} - \xi_{e_2} \dd \xi_{e_1})$ pointwise.
The Bernstein polynomials are computed once per point, using an in-place de Casteljau algorithm, and re-used to assemble all shape functions.
The inner evaluation loops are, up to field prefixes, the following, with \lstinline|cB|
the vector of scalar Bernstein values at the point:

\begin{lstlisting}
# full: B_α(λ(x)) · ψ
ω[i,w] = cB[α_ids[w]] * Ψ[w]

# trimmed: bare λ^α  ...  · φ(ξ; e1, e2)
bb     = nrm[w] * cB[α_ids[w]]
ω[i,w] = bb*λ[e1s[w]]*Je2[w] - bb*λ[e2s[w]]*Je1[w]
\end{lstlisting}

The array \lstinline|nrm| divides out the multinomial factor of Gridap's Bernstein basis, since the trimmed calculus of \sect{sec:trimmed} requires the bare monomials of Remark~\ref{rem:normalisation}.

\subsection{The rotation kernel}\label{sec:impl-kernel}

The implementation of the change-of-basis calculus of \sects{sec:full}{sec:trimmed} is summarised by the following two functions. Each takes the triplet
table \lstinline|entries|, its inverse dictionary \lstinline|idx|, a basis-function index \lstinline|w| and a
permutation \lstinline|π|, and returns the list of \lstinline|(coefficient, index)| pairs of the
pullback expansion. The full space implements Theorem~\ref{thm:U3},

\begin{lstlisting}
function rotate(entries, idx, w, π)                    # full space
  F, k, α = entries[w]
  πF = sort(π[F]);  πk = π[k];  πα = α[invperm(π)]
  if α[k] > 0 && πk == minimum(πF)                     # filter hit
    [(-1.0, idx[(πF, v, πα)]) for v in πF if v != πk]  # ≤ D terms, all −1
  else
    [(1.0, idx[(πF, πk, πα)])]                         # single term, +1
  end
end
\end{lstlisting}

and the trimmed space implements Theorem~\ref{thm:T3},

\begin{lstlisting}
function rotate(entries, idx, w, π)                    # trimmed space
  F, (e1, e2), α = entries[w]
  πF = sort(π[F]);  πe1, πe2 = π[e1], π[e2];  πα = α[invperm(π)]
  ε  = πe1 < πe2 ? 1.0 : -1.0                          # alternation sign ε(π(e))
  eπ1, eπ2 = minmax(πe1, πe2);  m = minimum(πF)        # π(e)↑ and min π(f)
  if m != eπ1                                          # filter hit: two shifts ρ
    α1 = copy(πα); α1[m] -= 1; α1[eπ1] += 1
    α2 = copy(πα); α2[m] -= 1; α2[eπ2] += 1
    [(ε, idx[(πF, (m, eπ2), α1)]), (-ε, idx[(πF, (m, eπ1), α2)])]
  else
    [(ε, idx[(πF, (eπ1, eπ2), πα)])]                   # single signed term
  end
end
\end{lstlisting}

The code shows the structure proved above: every coefficient is $\pm 1$, the only inversion is the permutation \lstinline|invperm|, the reverse pullback re-runs the same kernel at \lstinline|invperm(π)| (Theorems~\ref{thm:U4} and \ref{thm:T4}), and the hit test costs two integer comparisons in the full space and one in the trimmed space.

The transformation matrix $T(\pi)$ of Corollaries~\ref{cor:cob} and \ref{cor:trimmed-cob} need not be assembled. A small
cache object stores the list of rows produced by the kernel for encountered permutations. At most $(D+1)!$ different permutations may arise ($24$ in three
dimensions), so the cached data for all of them is a few kilobytes.

\begin{remark}[1-based indexing]
  Julia, and every Gridap
  interface underneath (Bernstein term identifiers, polytope face tables, vertex
  numbering), is 1-based, so the code labels the vertices $1, \dots, D+1$ instead of
  $0, \dots, D$. All label logic is order-theoretic (comparisons and minima), so the
  translation is the uniform shift by one.
\end{remark}

\subsection{Global assembly}\label{sec:impl-mesh}

Each space is implemented as a Gridap reference element on $\Delta$ with the rotating basis as shape functions. This is the only mesh-independent object; the rest of this section describes how every cell obtains its own conforming basis from it.

The rotated basis of a cell $K$ is given by \eqref{eq:rotated-in-reference} and its global DOF indices by \eqref{eq:identification}, \sect{sec:conformity}. The assembly loop of Gridap, as of many finite element codes, has the maps $\operatorname{ltg}^d_K$ at hand, indexed by the local faces of the reference polytope. It is therefore convenient to store the rotated basis of $K$ by \emph{local} face, reordering the faces by
\[ \Pi_K : (f, p) \mapsto (\pi_K^{-1}(f), p), \]
the block of the sorted face $f$ moved to the block of its local labelling $\pi_K^{-1}(f)$, positions untouched; the first condition of \eqref{eq:identification} is then read off $\operatorname{ltg}^d_K$ directly. Let $\Pi(\pi_K)$ be the corresponding permutation matrix, $\Pi(\pi_K)_{(\hat f, q), (f, p)} = 1$ if $\hat f = \pi_K^{-1}(f)$ and $q = p$, and $0$ otherwise. The global, conforming shape-function basis on cell $K$ is obtained from the reference rotating basis with
\[ M_K = \Pi(\pi_K) \, T(\pi_K^{-1}), \]
block diagonal by local face, where $T(\pi_K^{-1})$ is the matrix of \eqref{eq:rotated-in-reference}. $\Pi(\pi_K)$ only adapts the sorted-face block structure of $T(\pi_K^{-1})$ to the local-face convention of the reference element; an assembler indexing cell blocks by sorted face would use $T(\pi_K^{-1})$ alone.

The cell matrix of bi-linear forms on the rotated basis becomes
the congruence $M_K A_K M_K^{\top}$ of its matrix $A_K$ in the unrotated basis.
Although the spaces considered here admit no geometrically decomposed invariant basis for infinitely many degrees \cite{Licht2023symmetry,BerchenkoKogan2024}, the spanning
sets $\Sset(\Delta_\xi, r)$ are always invariant by relabelling (Lemma~\ref{lem:psi-equivariance}). The pullbacks
of basis functions hitting the filter are just single polynomials of the
spanning set. On the regular reference simplex the matrix of any bi-linear form
evaluated at pairs of functions of a relabelled basis is, up to permutation, a
submatrix of the matrix of that form on the spanning set $\Sset(\Delta_\xi, r)$. Using combinatorics
instead of floating point operations could again spare computations and
increase accuracy during assembly, although the pullbacks (or Piola maps) to
the physical element would still have to be accounted for. This strategy has been studied in \cite{kirby_fast_2011,kirby_low-complexity_2014,kirby_fast_2017}.

\section{Numerical experiments}\label{sec:numexp}

All experiments are run with Julia 1.12 and Gridap 0.20.10 which are both open
source under the MIT license. The
test suite is archived in \cite{BadiaManyerMarteau2026}. Throughout, ``full'' denotes the
rotating basis of $\PrL$ and ``trimmed'' that of $\PrmL$, and $n$ denotes the
dimension of the local basis.

The following experiments first validate the analytical formulas for the rotation
changes of basis, the correctness of the finite element implementation, and of
the conforming assembly of the global FE spaces. Then, the convergence rates for
$L^2$ projection and for a curl--curl problem are verified.

\subsection{Exactness and the representation property}\label{sec:exp-exact}

The first experiments validate Corollaries~\ref{cor:cob} and
\ref{cor:trimmed-cob} for both spaces, $D \in \{2, 3\}$ and $r \leq 4$. For
every $\pi, \tau \in S_{D+1}$, the entries of $T(\pi)$ are exactly $-1$, $0$ or $+1$, and the identities $T(\pi) \, T(\pi^{-1}) = I$ and $T(\pi \circ \tau) = T(\tau) \, T(\pi)$ hold exactly in integer arithmetic. The transformation law is validated against
numerically pulled-back basis functions at $20$ interior points up to machine
precision. The checks are run for all the $36$ $(\pi, \tau)$ pairs in two dimensions
and all the $576$ ones in three dimensions.

\subsection{Sparsity and cost}\label{sec:exp-sparsity}

The second experiment quantifies the sparsity statement of
Corollaries~\ref{cor:cob} and \ref{cor:trimmed-cob}. \tab{tab:sparsity} reports,
per space and degree, the average number of multi-term rows per vertex
permutation, the mean and maximum number of non-zeros per row, and the time to
produce the signed index map of one permutation, cold (first request) and hot
(memoised). Multi-term rows are absent below $r = 3$ in the full space, the first
degree with a fully supported exponent on a face of dimension at least two, and
below $r = 2$ in the trimmed space, the first degree at which such faces own
basis functions. Even at $r = 4$ in three dimensions the mean row carries at most
$1.27$ entries, and the density of $T$ decreases with $r$. A memoised lookup
costs about ten nanoseconds, and, at the degrees reported, the cold computation
of a full map is cheaper than materialising the same map as a dense $n \times n$
matrix, the operation \sect{sec:impl-mesh} performs once per permutation.

\begin{table}[htbp]
  \centering
  \small
  \begin{tabular}{@{}lcccccccc@{}}
    \toprule
    \textbf{space} & \textbf{D} & \textbf{r} & \textbf{n} & \textbf{multi-term rows} & \textbf{nnz/row} & \textbf{nnz vs $n^2$} & \textbf{map, cold} & \textbf{map, hot}\\
                   &            &            &            & \textbf{mean/$\pi$}      & \textbf{mean (max)} &                    &                    &                  \\
    \midrule
    full    & 2 & 2 & 12  & 0     & 1.00 (1) & 8.3\%  & 1.2 $\mu$s  & 10 ns\\
    full    & 2 & 3 & 20  & 0.67  & 1.03 (2) & 5.2\%  & 2.0 $\mu$s  & 10 ns\\
    full    & 2 & 4 & 30  & 2.00  & 1.07 (2) & 3.6\%  & 3.0 $\mu$s  & 10 ns\\
    full    & 3 & 3 & 60  & 2.67  & 1.04 (2) & 1.7\%  & 6.3 $\mu$s  & 11 ns\\
    full    & 3 & 4 & 105 & 8.75  & 1.09 (3) & 1.0\%  & 10.8 $\mu$s & 11 ns\\
    trimmed & 2 & 2 & 8   & 0.67  & 1.08 (2) & 13.5\% & 0.84 $\mu$s & 10 ns\\
    trimmed & 2 & 3 & 15  & 2.00  & 1.13 (2) & 7.6\%  & 1.6 $\mu$s  & 10 ns\\
    trimmed & 2 & 4 & 24  & 4.00  & 1.17 (2) & 4.9\%  & 2.6 $\mu$s  & 10 ns\\
    trimmed & 3 & 3 & 45  & 9.50  & 1.21 (2) & 2.7\%  & 5.0 $\mu$s  & 11 ns\\
    trimmed & 3 & 4 & 84  & 22.00 & 1.26 (2) & 1.5\%  & 9.7 $\mu$s  & 11 ns\\
    \bottomrule
  \end{tabular}
  \caption{Sparsity of the rotation maps, aggregated over all $(D+1)!$
    permutations, and the cost of producing the signed index map of one
    permutation (median timings, Intel Core i7-13800H). The degrees $r = 1$,
    whose maps are index permutations, signed in the trimmed space, and $r = 2$ in three dimensions
    are omitted.
    Single-entry $-1$ hits on edges occur from $r = 2$ in the full space.}
  \label{tab:sparsity}
\end{table}

\subsection{Conformity on scrambled meshes}\label{sec:exp-conformity}

On meshes whose cells order their vertices arbitrarily, the traces of the global
basis functions must match across every interior facet. We call a mesh
\emph{scrambled} when the local vertex ordering of each of its cells has been
arbitrarily permuted. Conformity is validated on two types of scrambled meshes,
two-cell ones and larger ones. The two-cell meshes run $18$ ordering pairs for
two triangles and six for two tetrahedra, including a non-cyclic permutation of
the shared face. The larger meshes are simplexified Cartesian grids in which
every cell's ordering is scrambled by a seeded random permutation.
\tab{tab:conformity} reports the maximum trace jump over all interior facets,
all global basis functions and all facet sample points, in the two settings. In
all cases, the trace jump is at machine precision with the rotation active and
$O(1)$ otherwise.

\begin{table}[htbp]
  \centering
  \footnotesize
  \begin{tabular}{@{}lcccccc@{}}
    \toprule
    \textbf{space} & \textbf{D} & \textbf{r} & \textbf{cells} & \textbf{dofs} & \textbf{jump, rotation on} & \textbf{jump, rotation off}\\
    \midrule
    full    & 2 & 1 & 2 / 128 & 416  & $2.8 \cdot 10^{-17}$ / $4.7 \cdot 10^{-16}$ & $1.0$ / $1.0$  \\
    full    & 2 & 2 & 2 / 128 & 1008 & $4.7 \cdot 10^{-17}$ / $8.9 \cdot 10^{-16}$ & $1.0$ / $1.0$  \\
    full    & 2 & 3 & 2 / 128 & 1856 & $6.0 \cdot 10^{-17}$ / $1.3 \cdot 10^{-15}$ & $0.75$ / $0.75$\\
    full    & 3 & 1 & 2 / 384 & 1208 & $5.6 \cdot 10^{-17}$ / $2.8 \cdot 10^{-16}$ & $0.90$ / $0.90$\\
    full    & 3 & 2 & 2 / 384 & 4404 & $1.1 \cdot 10^{-16}$ / $4.4 \cdot 10^{-16}$ & $0.60$ / $0.60$\\
    trimmed & 2 & 1 & 2 / 128 & 208  & $2.8 \cdot 10^{-17}$ / $8.9 \cdot 10^{-16}$ & $2.0$ / $2.0$  \\
    trimmed & 2 & 2 & 2 / 128 & 672  & $1.1 \cdot 10^{-16}$ / $1.2 \cdot 10^{-15}$ & $1.0$ / $1.0$  \\
    trimmed & 2 & 3 & 2 / 128 & 1392 & $2.4 \cdot 10^{-17}$ / $1.6 \cdot 10^{-15}$ & $0.75$ / $0.75$\\
    trimmed & 3 & 1 & 2 / 384 & 604  & $1.1 \cdot 10^{-16}$ / $4.4 \cdot 10^{-16}$ & $1.8$ / $1.8$  \\
    trimmed & 3 & 2 & 2 / 384 & 2936 & $1.1 \cdot 10^{-16}$ / $6.7 \cdot 10^{-16}$ & $0.81$ / $0.81$\\
    \bottomrule
  \end{tabular}
  \caption{Maximum trace jump across interior facets, two-cell meshes
    over the tested relative orderings, $18$ in 2D and six in 3D (first number), and fully scrambled simplexified
    Cartesian meshes, $8 \times 8$ in 2D and $4 \times 4 \times 4$ in 3D (second number; the DOF count refers to these),
    using the relabelling change of basis (rotation on) or not (rotation off).}
  \label{tab:conformity}
\end{table}

\subsection[\texorpdfstring{$L^2$}{L2} projection on scrambled meshes]{$L^2$ projection on scrambled meshes}\label{sec:exp-convergence}

The experiment tests the convergence of the $L^2$ projection of a smooth 1-form
with non-vanishing exterior derivative,
\[ u = \sin(\pi x) \cos(\pi y) \, \dd x^1 - \cos(\pi x) \sin(\pi y) \, \dd x^2, \]
on the unit square. Both full and trimmed
spaces are used on simplexified Cartesian meshes with $4$ to $32$ cells per
side. Every mesh is assembled twice from the same connectivity: once with all
cell vertex lists sorted, so that no rotation is needed and it is automatically
skipped, and once with every cell's ordering scrambled by a seeded random
permutation. \tab{tab:convergence} reports the $L^2$ errors and rates on the
scrambled meshes, together with the $L^2$ norm of the difference between the
sorted-mesh and scrambled-mesh discrete solutions. The rates are the expected
$r + 1$ for the full space and $r$ for the trimmed space, and the two discrete
solutions coincide to $10^{-15}$. The curl seminorm of the projection error,
which the $L^2$ projection does not control, converges at rate $r$ (full) and
between $r - 1$ and $r$ (trimmed). In three dimensions (up to $8^3 \times 6$
cells, $r \leq 2$) the affordable levels are pre-asymptotic but approach the
same rates ($1.97$ and $2.75$ full, $0.91$ and $1.92$ trimmed at the finest
pairs), with the same machine-zero sorted-versus-scrambled agreement.

\begin{table}[htbp]
  \centering
  \small
  \begin{tabular}{@{}lccccccc@{}}
    \toprule
    \textbf{space} & \textbf{r} & $\mathbf{8}$ & $\mathbf{16}$ & $\mathbf{32}$ & \textbf{rate} & \textbf{expected} & \textbf{sorted vs scrambled}\\
    \midrule
    full    & 1 & $8.97 \cdot 10^{-3}$ & $2.25 \cdot 10^{-3}$ & $5.63 \cdot 10^{-4}$ & $2.0$ & $2$ & $3.9 \cdot 10^{-16}$\\
    full    & 2 & $5.25 \cdot 10^{-4}$ & $6.71 \cdot 10^{-5}$ & $8.45 \cdot 10^{-6}$ & $3.0$ & $3$ & $5.1 \cdot 10^{-16}$\\
    full    & 3 & $1.95 \cdot 10^{-5}$ & $1.23 \cdot 10^{-6}$ & $7.71 \cdot 10^{-8}$ & $4.0$ & $4$ & $7.4 \cdot 10^{-16}$\\
    trimmed & 1 & $8.00 \cdot 10^{-2}$ & $4.01 \cdot 10^{-2}$ & $2.00 \cdot 10^{-2}$ & $1.0$ & $1$ & $1.7 \cdot 10^{-16}$\\
    trimmed & 2 & $4.06 \cdot 10^{-3}$ & $1.01 \cdot 10^{-3}$ & $2.53 \cdot 10^{-4}$ & $2.0$ & $2$ & $3.8 \cdot 10^{-16}$\\
    trimmed & 3 & $1.83 \cdot 10^{-4}$ & $2.32 \cdot 10^{-5}$ & $2.91 \cdot 10^{-6}$ & $3.0$ & $3$ & $6.4 \cdot 10^{-16}$\\
    \bottomrule
  \end{tabular}
  \caption{$L^2$-projection errors $\|u - u_h\|_{L^2}$ on fully scrambled
    2D meshes with the given number of cells per side, observed rate between the two finest
    levels, and, in the last column, the maximum over all levels of the $L^2$ norm of the difference
    between the sorted-mesh and scrambled-mesh discrete solutions.}
  \label{tab:convergence}
\end{table}

\subsection{A curl--curl problem on scrambled meshes}\label{sec:exp-curlcurl}

Lastly, the definite Maxwell problem $\curl \curl u + u = f$ is solved on the
unit square and the unit cube with the natural boundary condition $n \times
\curl u = 0$, using the weak form $\int_\Omega \langle \curl u, \curl v \rangle +
\int_\Omega \langle u, v \rangle = \int_\Omega \langle f, v \rangle$ for all $v$
in the conforming space.
The exact solutions are chosen so that the boundary condition holds
exactly and $f$ is a multiple of $u$, namely
\[
  u = \sin(\pi x) \cos(\pi y) \, \dd x^1 - \cos(\pi x) \sin(\pi y) \, \dd x^2, \qquad f = (1 + 2\pi^2) u,
\]
in 2D, and
\[
  u = -\sin(\pi x) \cos(\pi y) \cos(\pi z) \, \dd x^1 + \cos(\pi x) \sin(\pi y) \cos(\pi z) \, \dd x^2, \qquad f = (1 + 3\pi^2) u,
\]
in 3D. The meshes, the sorted and scrambled orderings and
the comparison of the two discrete solutions are those of \sect{sec:exp-convergence}.
\tab{tab:curlcurl} reports the errors on the finest level, the rate between the two
finest levels, and the maximum over all levels of the $L^2$ norm of the difference
between the sorted-mesh and scrambled-mesh discrete solutions. In 2D the $L^2$
rates are $r + 1$ (full) and $r$ (trimmed) and the curl-seminorm rates are $r$ for
both spaces, to two digits. In 3D, with $2$ to $8$ cells per side, the rates are
pre-asymptotic and approach the same values. On every mesh the two discrete
solutions coincide to $10^{-12}$ or better, the same linear system being assembled
through different local orderings and solved by the same direct solver.

\begin{table}[htbp]
  \centering
  \footnotesize\setlength{\tabcolsep}{4.5pt}
  \begin{tabular}{@{}lcccccccccc@{}}
    \toprule
    \textbf{space} & \textbf{D} & \textbf{r} & \textbf{cells/side} & \textbf{dofs} & $\|u - u_h\|_{L^2}$ & \textbf{rate} & $\|\dd(u - u_h)\|_{L^2}$ & \textbf{rate} & \textbf{expected} & \textbf{sorted vs scrambled}\\
    \midrule
    full    & 2 & 1 & 32 & 6272  & $9.36 \cdot 10^{-4}$ & $2.0$  & $1.03 \cdot 10^{-1}$ & $1.0$  & $2$, $1$ & $6.1 \cdot 10^{-13}$\\
    full    & 2 & 2 & 32 & 15552 & $9.48 \cdot 10^{-6}$ & $3.0$  & $1.95 \cdot 10^{-3}$ & $2.0$  & $3$, $2$ & $1.7 \cdot 10^{-12}$\\
    full    & 2 & 3 & 32 & 28928 & $9.44 \cdot 10^{-8}$ & $4.0$  & $2.71 \cdot 10^{-5}$ & $3.0$  & $4$, $3$ & $3.6 \cdot 10^{-12}$\\
    trimmed & 2 & 1 & 32 & 3136  & $2.00 \cdot 10^{-2}$ & $1.0$  & $1.03 \cdot 10^{-1}$ & $1.0$  & $1$, $1$ & $2.2 \cdot 10^{-13}$\\
    trimmed & 2 & 2 & 32 & 10368 & $2.80 \cdot 10^{-4}$ & $2.0$  & $1.95 \cdot 10^{-3}$ & $2.0$  & $2$, $2$ & $1.3 \cdot 10^{-12}$\\
    trimmed & 2 & 3 & 32 & 21696 & $3.06 \cdot 10^{-6}$ & $3.0$  & $2.71 \cdot 10^{-5}$ & $3.0$  & $3$, $3$ & $3.9 \cdot 10^{-12}$\\
    full    & 3 & 1 & 8  & 8368  & $1.89 \cdot 10^{-2}$ & $1.85$ & $4.81 \cdot 10^{-1}$ & $0.94$ & $2$, $1$ & $1.4 \cdot 10^{-13}$\\
    full    & 3 & 2 & 8  & 32136 & $7.72 \cdot 10^{-4}$ & $2.94$ & $4.44 \cdot 10^{-2}$ & $1.93$ & $3$, $2$ & $3.4 \cdot 10^{-13}$\\
    trimmed & 3 & 1 & 8  & 4184  & $7.90 \cdot 10^{-2}$ & $0.95$ & $4.81 \cdot 10^{-1}$ & $0.94$ & $1$, $1$ & $6.9 \cdot 10^{-14}$\\
    trimmed & 3 & 2 & 8  & 21424 & $5.71 \cdot 10^{-3}$ & $1.95$ & $4.44 \cdot 10^{-2}$ & $1.93$ & $2$, $2$ & $1.6 \cdot 10^{-13}$\\
    \bottomrule
  \end{tabular}
  \caption{Curl--curl problem on fully scrambled simplexified Cartesian meshes.
    Errors on the finest level ($32$ cells per side in 2D, $8$ in 3D), rates between
    the two finest levels ($4$ to $32$ cells per side in 2D, $2$ to $8$ in 3D),
    expected rates for the $L^2$ and curl-seminorm errors, and the maximum over all
    levels of the $L^2$ norm of the difference between the sorted-mesh and
    scrambled-mesh discrete solutions.}
  \label{tab:curlcurl}
\end{table}

\section{Conclusions and future work}\label{sec:conclusions}

We have constructed new geometrically decomposed bases of $\PrL$ on
simplices, and shown that these and the standard bases of $\PrmL$
are well behaved under vertex relabelling.
Most basis polynomials pull back
to a single signed basis polynomial of the basis defined on
the relabelled coordinates, except on an
explicitly characterised hit set. In such cases, resummation and shift identities
express the pullback as a degree-independent
signed sum of at most $D$ (full), respectively two (trimmed),
terms. The transformation matrices that achieve conformity of the shape-function and DOF bases are given by closed-form combinatorial formulas that are cheap to compute and exact in integer arithmetic, and the implementation confirms this at machine precision.

Although the explicit expansion appears not to have been recorded, the pullback
relabelling theorems apply to the geometrically decomposed bases from
\cite{ArnoldFalkWinther2009} as well as those we introduced. However, the new
rotating bases and their relabelling logic are cheaper to implement. The reason is the definition of the directional 1-forms through the support indicator of the multi-exponent rather than the multi-exponent itself, which gives a degree-independent catalogue of forms and transformations.

The extension to $K$-forms is left for future work. Directional $K$-forms are wedge products of catalogue 1-forms and pullback commutes with the wedge, so the structure of the change of basis by rotation is expected to carry over. The anchoring filter for $K \geq 2$ constrains the whole index set of the wedge, so a filter hit may require several resummations at once instead of the single one of Lemma~\ref{lem:psi-resummation}.

\section*{Acknowledgements}
This research was partially funded by the Australian Government through the Australian Research Council (project number DP220103160). 

\section*{Reproducibility}
The software used to produce the numerical results in this paper is archived in \cite{BadiaManyerMarteau2026}.

\appendix
\section{Bernstein normalisation of the trimmed basis}\label{sec:bernstein}

This section records the counterpart of \sect{sec:trimmed} for the Bernstein
normalisation of Remark~\ref{rem:normalisation}. Write
\begin{equation}\label{eq:bernstein-scaling}
  c(\alpha) \doteq |\alpha|! / (\alpha_0 ! \cdots \alpha_D !), \quad
  \tilde{\BB}(\xi; \alpha) \doteq c(\alpha) \, \BB(\xi; \alpha), \quad
  \tilde{w}(\xi; f, e, \alpha) \doteq c(\alpha) \, w(\xi; f, e, \alpha),
\end{equation}
so that $\tilde{w}$ is the basis function \eqref{eq:trimmed-w} with the Bernstein scalar
factor. The index sets $\Sset^-(\Delta_\xi, r)$ and $\Bset^-(\Delta_\xi, r)$ are unchanged, and so
are Theorem~\ref{thm:trimmed-spanning}, Definition~\ref{def:trimmed-filter}, Theorem~\ref{thm:T1} and Lemma~\ref{lem:trimmed-count} and
Corollary~\ref{cor:trimmed-basis}, whose statements are independent of the value of the scalar
factor. Theorem~\ref{thm:phi-equivariance} and Corollary~\ref{cor:trimmed-sparsity} are unchanged as well,
since the catalogue of Whitney forms does not depend on $\alpha$. Only Theorem~\ref{thm:T2},
Theorem~\ref{thm:T3} and Corollary~\ref{cor:trimmed-cob} change, because the shift $\rho$ transfers one unit of
degree between entries of $\alpha$.

\begin{lemma}[shift ratio]\label{lem:shift-ratio}
  For distinct labels $i \neq j$ and a multi-exponent $\alpha$ with $\alpha_i > 0$,
  \begin{equation}\label{eq:shift-ratio}
    c(\alpha) / c(\rho(i, j) \circ \alpha) = (\alpha_j + 1) / \alpha_i.
  \end{equation}
\end{lemma}
\begin{proof}
  The shift preserves $|\alpha|$, and replaces $\alpha_i !$ by $(\alpha_i - 1)!$ and $\alpha_j !$ by
  $(\alpha_j + 1)!$.
\end{proof}

\begin{corollary}[two-term decomposition, Bernstein]\label{cor:bernstein-decomp}
  With the notation of Theorem~\ref{thm:T2},
  \begin{equation}\label{eq:bernstein-decomp}
    \begin{aligned}
      \tilde{w}(\xi; f, (e_1, e_2), \alpha) =
        &\ \frac{\alpha_{e_1} + 1}{\alpha_m} \, \tilde{w}(\xi; f, (m, e_2), \rho(m, e_1) \circ \alpha)\\
      - &\ \frac{\alpha_{e_2} + 1}{\alpha_m} \, \tilde{w}(\xi; f, (m, e_1), \rho(m, e_2) \circ \alpha).
    \end{aligned}
  \end{equation}
\end{corollary}
\begin{proof}
  Multiply \eqref{eq:trimmed-decomp} by $c(\alpha)$ and use Lemma~\ref{lem:shift-ratio} at $(m, e_1)$
  and $(m, e_2)$, both admissible because $\alpha_m > 0$.
\end{proof}

\begin{corollary}[pullback of the basis functions, Bernstein]\label{cor:bernstein-pullback}
  With the notation of Theorem~\ref{thm:T3}, write $\beta = \pi(\alpha)$. The relabelled case is
  \[ A_{\pi^{-1}}^* (\tilde{w}(\xi; f, e, \alpha)) = \varepsilon \, \tilde{w}(\lambda; \pi(f), \pi(e)^{\sort}, \beta), \]
  and the filter hit is
  \begin{equation}\label{eq:bernstein-pullback-hit}
    \begin{aligned}
      A_{\pi^{-1}}^* (\tilde{w}(\xi; f, e, \alpha)) = \varepsilon \bigl(
        &\ \frac{\beta_{e_1^\pi} + 1}{\beta_{m^\pi}} \,
          \tilde{w}(\lambda; \pi(f), (m^\pi, e_2^\pi), \rho(m^\pi, e_1^\pi) \circ \beta)\\
      - &\ \frac{\beta_{e_2^\pi} + 1}{\beta_{m^\pi}} \,
          \tilde{w}(\lambda; \pi(f), (m^\pi, e_1^\pi), \rho(m^\pi, e_2^\pi) \circ \beta) \bigr).
    \end{aligned}
  \end{equation}
\end{corollary}
\begin{proof}
  The multinomial factor $c$ is invariant under any permutation of the entries of
  $\alpha$, so $c(\pi(\alpha)) = c(\alpha)$ and \eqref{eq:trimmed-pullback-raw} holds with $\tilde{\BB}$ and
  $\tilde{w}$ in place of $\BB$ and $w$. Which of the two cases holds is a property
  of the pair, as in Theorem~\ref{thm:T3}. The hit case then results from applying
  Corollary~\ref{cor:bernstein-decomp} on $\Delta_\lambda$ to $(\pi(f), \pi(e)^{\sort}, \beta)$.
\end{proof}

\begin{corollary}[transformation matrix, Bernstein]\label{cor:bernstein-cob}
  Define $\tilde{T}(\pi)$ as in Corollary~\ref{cor:trimmed-cob} with $\tilde{w}$ in place of $w$, and
  let $S$ be the diagonal matrix whose entry at $\mu = (f, e, \alpha)$ is $c(\alpha)$. Then
  \begin{equation}\label{eq:bernstein-cob}
    \tilde{T}(\pi) = S \, T(\pi) \, S^{-1}.
  \end{equation}
  Consequently, $\tilde{T}(\pi)$ has the sparsity of $T(\pi)$, and $\tilde{T}(\pi)^{-1}
  = \tilde{T}(\pi^{-1})$.
\end{corollary}
\begin{proof}
  $\tilde{w}_\mu = c(\alpha) \, w_\mu$ on both $\Delta_\xi$ and $\Delta_\lambda$, so \eqref{eq:bernstein-cob}
  restates the expansion of Corollary~\ref{cor:trimmed-cob}. A diagonal change of basis
  preserves the sparsity and the identity $T(\pi)^{-1} = T(\pi^{-1})$. The entries
  follow from Corollary~\ref{cor:bernstein-pullback}.
\end{proof}

Values in a row of $\tilde{T}(\pi)$ need no longer be equal in absolute value, and may be
rational instead of integer.

\printbibliography

@article{ArnoldFalkWinther2006,
  author  = {Arnold, Douglas N. and Falk, Richard S. and Winther, Ragnar},
  title   = {Finite element exterior calculus, homological techniques, and applications},
  journal = {Acta Numerica},
  volume  = {15},
  pages   = {1--155},
  year    = {2006},
  doi     = {10.1017/S0962492906210018}
}

@article{ArnoldFalkWinther2009,
  author  = {Arnold, Douglas N. and Falk, Richard S. and Winther, Ragnar},
  title   = {Geometric decompositions and local bases for spaces of finite element differential forms},
  journal = {Computer Methods in Applied Mechanics and Engineering},
  volume  = {198},
  number  = {21--26},
  pages   = {1660--1672},
  year    = {2009},
  doi     = {10.1016/j.cma.2008.12.017},
  note    = {arXiv:0806.1255}
}

@article{Licht2022,
  author  = {Licht, Martin W.},
  title   = {On basis constructions in finite element exterior calculus},
  journal = {Advances in Computational Mathematics},
  volume  = {48},
  number  = {2},
  pages   = {14},
  year    = {2022},
  doi     = {10.1007/s10444-022-09926-6},
  note    = {arXiv:1810.01896}
}

@article{Licht2023symmetry,
  author  = {Licht, Martin W.},
  title   = {Symmetry and invariant bases in finite element exterior calculus},
  journal = {Foundations of Computational Mathematics},
  volume  = {24},
  number  = {4},
  pages   = {1185--1224},
  year    = {2024},
  doi     = {10.1007/s10208-023-09609-8},
  note    = {arXiv:1912.11002}
}

@article{BerchenkoKogan2024,
  author  = {Berchenko-Kogan, Yakov},
  title   = {Symmetric bases for finite element exterior calculus spaces},
  journal = {Foundations of Computational Mathematics},
  volume  = {24},
  pages   = {1485--1515},
  year    = {2024},
  doi     = {10.1007/s10208-023-09617-8},
  note    = {arXiv:2112.06065}
}

@article{BerchenkoKogan2025extension,
  author  = {Berchenko-Kogan, Yakov},
  title   = {Extension operators and geometric decompositions},
  journal = {ESAIM: Mathematical Modelling and Numerical Analysis},
  volume  = {60},
  number  = {5},
  pages   = {2269--2283},
  year    = {2026},
  doi     = {10.1051/m2an/2026061},
  note    = {arXiv:2505.00129}
}

@article{ScroggsDokkenRichardsonWells2022,
  author  = {Scroggs, Matthew W. and Dokken, J{\o}rgen S. and Richardson, Chris N. and Wells, Garth N.},
  title   = {Construction of arbitrary order finite element degree-of-freedom maps on polygonal and polyhedral cell meshes},
  journal = {ACM Transactions on Mathematical Software},
  volume  = {48},
  number  = {2},
  pages   = {18:1--18:23},
  year    = {2022},
  doi     = {10.1145/3524456},
  note    = {arXiv:2102.11901}
}

@misc{ScroggsWells2026,
  author = {Scroggs, Matthew W. and Wells, Garth N.},
  title  = {Algorithm {XXXX}: Computation of finite element degree-of-freedom transformation matrices},
  year   = {2026},
  note   = {arXiv:2607.08172, submitted to ACM Transactions on Mathematical Software}
}

@article{RognesKirbyLogg2009,
  author  = {Rognes, Marie E. and Kirby, Robert C. and Logg, Anders},
  title   = {Efficient assembly of {H(div)} and {H(curl)} conforming finite elements},
  journal = {SIAM Journal on Scientific Computing},
  volume  = {31},
  number  = {6},
  pages   = {4130--4151},
  year    = {2009},
  doi     = {10.1137/08073901X}
}

@article{Kirby2018,
  author  = {Kirby, Robert C.},
  title   = {A general approach to transforming finite elements},
  journal = {SMAI Journal of Computational Mathematics},
  volume  = {4},
  pages   = {197--224},
  year    = {2018},
  doi     = {10.5802/smai-jcm.33}
}

@article{KirbyMitchell2019,
  author  = {Kirby, Robert C. and Mitchell, Lawrence},
  title   = {Code generation for generally mapped finite elements},
  journal = {ACM Transactions on Mathematical Software},
  volume  = {45},
  number  = {4},
  pages   = {41:1--41:23},
  year    = {2019},
  doi     = {10.1145/3361745},
  note    = {arXiv:1808.05513}
}

@article{FuentesKeithDemkowiczNagaraj2015,
  author  = {Fuentes, Federico and Keith, Brendan and Demkowicz, Leszek and Nagaraj, Sriram},
  title   = {Orientation embedded high order shape functions for the exact sequence elements of all shapes},
  journal = {Computers \& Mathematics with Applications},
  volume  = {70},
  number  = {4},
  pages   = {353--458},
  year    = {2015},
  doi     = {10.1016/j.camwa.2015.04.027},
  note    = {arXiv:1504.03025}
}

@phdthesis{Zaglmayr2006,
  author = {Zaglmayr, Sabine},
  title  = {High order finite element methods for electromagnetic field computation},
  school = {Johannes Kepler University Linz},
  year   = {2006}
}

@article{AinsworthCoyle2003,
  author  = {Ainsworth, Mark and Coyle, Joe},
  title   = {Hierarchic finite element bases on unstructured tetrahedral meshes},
  journal = {International Journal for Numerical Methods in Engineering},
  volume  = {58},
  number  = {14},
  pages   = {2103--2130},
  year    = {2003},
  doi     = {10.1002/nme.847}
}

@article{AgelekAndersonBangerthBarth2017,
  author  = {Agelek, Rainer and Anderson, Michael and Bangerth, Wolfgang and Barth, William L.},
  title   = {On orienting edges of unstructured two- and three-dimensional meshes},
  journal = {ACM Transactions on Mathematical Software},
  volume  = {44},
  number  = {1},
  pages   = {5:1--5:22},
  year    = {2017},
  doi     = {10.1145/3061708}
}

@article{KinnewigWickBeuchler2025,
  author  = {Kinnewig, Sebastian and Wick, Thomas and Beuchler, Sven},
  title   = {Algorithmic realization of the solution to the sign conflict problem for hanging nodes on hp-hexahedral {N}{\'e}d{\'e}lec elements},
  journal = {ACM Transactions on Mathematical Software},
  volume  = {51},
  number  = {4},
  pages   = {1--20},
  year    = {2025},
  doi     = {10.1145/3766903},
  note    = {arXiv:2306.01416}
}

@article{RapettiBossavit2009,
  author  = {Rapetti, Francesca and Bossavit, Alain},
  title   = {Whitney forms of higher degree},
  journal = {SIAM Journal on Numerical Analysis},
  volume  = {47},
  number  = {3},
  pages   = {2369--2386},
  year    = {2009},
  doi     = {10.1137/070705489}
}

@article{Bossavit2002,
  author  = {Bossavit, Alain},
  title   = {Generating {W}hitney forms of polynomial degree one and higher},
  journal = {IEEE Transactions on Magnetics},
  volume  = {38},
  number  = {2},
  pages   = {341--344},
  year    = {2002},
  doi     = {10.1109/20.996092}
}

@article{ChristiansenRapetti2016,
  author  = {Christiansen, Snorre H. and Rapetti, Francesca},
  title   = {On high order finite element spaces of differential forms},
  journal = {Mathematics of Computation},
  volume  = {85},
  number  = {298},
  pages   = {517--548},
  year    = {2016},
  doi     = {10.1090/mcom/2995}
}

@article{LohiKettunen2021,
  author  = {Lohi, Jonni and Kettunen, Lauri},
  title   = {Whitney forms and their extensions},
  journal = {Journal of Computational and Applied Mathematics},
  volume  = {393},
  pages   = {113520},
  year    = {2021},
  doi     = {10.1016/j.cam.2021.113520}
}

@article{Lohi2022,
  author  = {Lohi, Jonni},
  title   = {Systematic implementation of higher order {W}hitney forms in methods based on discrete exterior calculus},
  journal = {Numerical Algorithms},
  volume  = {91},
  number  = {3},
  pages   = {1261--1285},
  year    = {2022},
  doi     = {10.1007/s11075-022-01301-2}
}

@article{BadiaVerdugo2020,
  author  = {Badia, Santiago and Verdugo, Francesc},
  title   = {Gridap: an extensible finite element toolbox in {J}ulia},
  journal = {Journal of Open Source Software},
  volume  = {5},
  number  = {52},
  pages   = {2520},
  year    = {2020},
  doi     = {10.21105/joss.02520}
}

@article{VerdugoBadia2022,
  author  = {Verdugo, Francesc and Badia, Santiago},
  title   = {The software design of {G}ridap: a finite element package based on the {J}ulia {JIT} compiler},
  journal = {Computer Physics Communications},
  volume  = {276},
  pages   = {108341},
  year    = {2022},
  doi     = {10.1016/j.cpc.2022.108341}
}

@article{MFEM2021,
  author  = {Anderson, Robert and Andrej, Julian and Barker, Andrew and others},
  title   = {{MFEM}: A modular finite element methods library},
  journal = {Computers \& Mathematics with Applications},
  volume  = {81},
  pages   = {42--74},
  year    = {2021},
  doi     = {10.1016/j.camwa.2020.06.009}
}

@article{kirby_fast_2011,
  title = {Fast simplicial finite element algorithms using {Bernstein} polynomials},
  volume = {117},
  issn = {0029-599X, 0945-3245},
  url = {http://link.springer.com/10.1007/s00211-010-0327-2},
  doi = {10.1007/s00211-010-0327-2},
  number = {4},
  journal = {Numerische Mathematik},
  author = {Kirby, Robert C.},
  month = apr,
  year = {2011},
  pages = {631--652},
}

@article{kirby_low-complexity_2014,
  title = {Low-complexity finite element algorithms for the de {Rham} complex on simplices},
  volume = {36},
  issn = {1064-8275, 1095-7197},
  url = {http://epubs.siam.org/doi/10.1137/130927693},
  doi = {10.1137/130927693},
  number = {2},
  journal = {SIAM Journal on Scientific Computing},
  author = {Kirby, Robert C.},
  month = jan,
  year = {2014},
  pages = {A846--A868},
}

@article{kirby_fast_2017,
  title = {Fast inversion of the simplicial {Bernstein} mass matrix},
  volume = {135},
  issn = {0945-3245},
  url = {https://doi.org/10.1007/s00211-016-0795-0},
  doi = {10.1007/s00211-016-0795-0},
  number = {1},
  journal = {Numerische Mathematik},
  author = {Kirby, Robert C.},
  month = jan,
  year = {2017},
  pages = {73--95},
}

@software{BadiaManyerMarteau2026,
  author       = {Badia, Santiago and Manyer, Jordi and Marteau, Antoine},
  title        = {Numerical experiments for the article titled ``{Rotating Bases for Finite Element Exterior Calculus: Closed-Form Change of Basis Under Vertex Permutations}''},
  year         = {2026},
  organization = {Zenodo},
  doi          = {10.5281/zenodo.23049283}
}

\end{document}